\documentclass[ejs]{imsart}

\RequirePackage[OT1]{fontenc}
\RequirePackage{amsthm,amsmath,natbib}
\RequirePackage[colorlinks,citecolor=blue,urlcolor=blue]{hyperref}

\usepackage{times}
\usepackage{bm}

\usepackage[plain,noend]{algorithm2e}

\usepackage[utf8x]{inputenc}
\usepackage{mathrsfs}
\usepackage{amssymb}
\usepackage{amsfonts}
\usepackage{wrapfig}
\usepackage{multirow}
\usepackage{longtable}
\usepackage{enumerate}
\usepackage{color, graphicx}

\newcommand{\BR}{\mathbb{R}}

\newcommand{\mX}{{\cal X}}

\def\boxit#1{\vbox{\hrule\hbox{\vrule\kern6pt
			\vbox{\kern6pt#1\kern6pt}\kern6pt\vrule}\hrule}}
\arxiv{arXiv:0000.0000}

\startlocaldefs
\numberwithin{equation}{section}
\theoremstyle{plain}
\newtheorem{theorem}{Theorem}[section]
\newtheorem{lemma}{Lemma}[section]
\newtheorem{corollary}{Corollary}[section]
\theoremstyle{remark}

\endlocaldefs

\begin{document}

\begin{frontmatter}
\title{Bayesian Shrinkage towards Sharp Minimaxity}
\runtitle{Bayesian Sharp Minimax}

\begin{aug}
\author{\fnms{Qifan} \snm{Song}\thanksref{a,e1}\ead[label=e1,mark]{qfsong@purdue.edu}}

\address[a]{Department of Statistics, Purdue University, \\ 250 University St, West Lafayette, Indiana, U.S.A.
\printead{e1}}

\runauthor{Q. Song}

\affiliation{Purdue University}

\end{aug}

\begin{abstract}
Shrinkage prior are becoming more and more popular in Bayesian modeling for high dimensional sparse problems due to its computational
efficiency. Recent works show that a polynomially decaying prior leads to satisfactory posterior asymptotics under regression
models. In the literature, statisticians have investigated how the global shrinkage parameter, i.e., the scale parameter, 
in a heavy tail prior affects the posterior contraction. In this work, we explore how the shape of the prior, or more specifically, the polynomial order of the prior tail
affects the posterior. We discover that, under the sparse normal means models, the polynomial order does affect the multiplicative constant
of the posterior contraction rate. More importantly, if the polynomial order is sufficiently close to 1, it will induce the 
optimal Bayesian posterior convergence, in the sense that the Bayesian contraction rate is sharply minimax, i.e., not only the order, but also the multiplicative constant of the posterior contraction rate are optimal.
The above Bayesian sharp minimaxity holds when the global shrinkage parameter follows a deterministic choice which
depends on the unknown sparsity $s$. Therefore, a Beta-prior modeling is further proposed, such that our sharply minimax Bayesian procedure
is adaptive to unknown $s$. Our theoretical discoveries are justified by simulation studies.
\end{abstract}


\end{frontmatter}

\section{Introduction}\label{intro}
In Bayesian inference for high dimensional sparse models, the prior distribution needs to
incorporate certain {\it a priori} knowledge of the structural sparsity.
The classical spike-and-slab modeling assigns two-groups prior to each entry of a sparse
$n$-dimensional parameter vector $\theta^{(n)}=(\theta_1,\dots,\theta_n)^T$, i.e., $\pi(\theta_i)$ is a mixture of two distributions which correspond to $\theta_i=0$ and $\theta_i\neq0$ respectively. 
This natural modeling however requires expensive posterior computation.
An alternative one-group Bayesian modeling, or so-called shrinkage prior, is much more computational
attractive, hence gains more and more popularity in the Bayesian community.

In this work, we consider the sparse normal means model, $y^{(n)}=\theta^{(n)}+\epsilon$, where 
$y^{(n)}=(y_1,\dots,y_n)^T\in\BR^n$, $\epsilon\sim N(0,\sigma^2 I_n)$, and the parameter of interest is
$\theta^{(n)}\in\BR^n$. 
Throughout this work, the variance $\sigma^2$ of the error term is assumed to be known, and W.O.L.G, we let $\sigma^2=1$.
Suppose that the true parameter value $\theta^*$ is a sparse
vector with $s^{(n)}$ nonzero entries, and asymptotically, we allow $s^{(n)}$ to increase simultaneously
with $n$.
For the sake of simplicity of the notation, we drop the superscript ${(n)}$ from $y^{(n)}$,
$\theta^{(n)}$ and $s^{(n)}$ in what follows.
This normal means model can be viewed as the simplest regression problem under orthogonality.
Theoretically, the study of the sparse normal means model can provide us important insights about Bayesian regression
under a shrinkage prior. In practice, applications of normal means model include multiple testing problems where $y_i$ is 
the $z$-test statistics, and signal detection problems where $y_i$ can be a noisy pixel in a 
functional magnetic resonance imaging (fMRI).
In the literature, a variety of shrinkage priors have been proposed for Bayesian
inferences of $\theta$. Popular examples include Bayesian Lasso \citep{ParkC2008,Hans2009},
Horseshoe prior \citep{CarvalhoPS2010}, Dirichlet-Laplace prior \citep{BhattacharyaPPD2015},
Normal-Exponential-Gamma distribution \citep{GriffinB2011},
Generalized double Pareto distribution \citep{ArmaganDL2013},
Generalized Beta mixture of Gaussian distributions \citep{ArmaganCD2011} and etc.
A general form of shrinkage priors can be written as 
\begin{equation}\label{gprior}
\pi(\theta|\tau) =\prod_{i=1}^n \{(1/\tau)\pi_0(\theta_i/\tau)\},\quad \tau\sim \pi(\tau),
\end{equation}
where the scale parameter $\tau$ is called global shrinkage parameter which controls the overall
shrinkage effect. $\tau$ can either follow a prior distribution as in (\ref{gprior}) or have a deterministic value. If furthermore, $\pi_0$ can be expressed as a scaled mixture of Gaussian
distribution, (\ref{gprior}) then leads to the so-called local-global shrinkage:
$\theta_i\sim N(0, \lambda^2_i\tau^2)$, $\lambda_i^2\sim\pi(\lambda_i^2), \tau\sim \pi(\tau)$
and $\lambda_i$'s are called local shrinkage parameters. Note that the Dirichlet-Laplace prior is an exception
and doesn't fit the general form (\ref{gprior}).

Given a wide choice of shrinkage priors, certain criteria are necessary to evaluate and compare 
these different priors, e.g., computational efficiency and theoretical convergence.
A benchmark for the theoretical performance is the posterior contraction rate, which characterizes how fast the posterior distribution (not just the Bayes estimator)
converges to the true parameter. For example, the $L_2$ posterior contraction rate $r_n$ satisfies:
\[
\lim_{n\rightarrow\infty} E^*[\pi(\|\theta-\theta^*\|\geq r_n|y)] = 0, \mbox{ for any }\theta^*,
\]
where $E^*$ denotes the expectation with respect to the data generation measure of $n$ dimensional data $y$ under true parameter $\theta^*$.
It is well known that the frequentist minimax rate for normal means models is 
$\min_{\widehat\theta}\max_{\theta^*}\|\widehat\theta-\theta^*\|=\{(2+o(1))s\log(n/s)\}^{1/2}$ \citep{DonohoJHS1992},
and many Bayesian works show that the Bayesian contraction rates are comparable to this minimax rate. 
For example, Dirichlet-Laplace prior \citep{BhattacharyaPPD2015} (under certain conditions of $\theta^*$) achieves $r_n\asymp \{s\log(n/s)\}^{1/2}$, and 
\citet{VanKV2014} showed that horseshoe prior achieves $r_n=M_n\{s\log(n/s)\}^{1/2}$ for any $M_n\rightarrow\infty$.
Furthermore, recent works \citep[e.g.,][]{van2016conditions, GhoshC2014,SongL2017} show that the tail behavior of $\pi_0$ plays an important role for posterior asymptotics, and suggest to choose a polynomial decaying $\pi_0$ in order to achieve (near-) optimal posterior contraction rate. It is worth noting that all the aforementioned shrinkage priors, except Dirichlet-Laplace prior, have a polynomial tail. Thus we believe polynomially decaying shrinkage priors are, generally speaking, 
good choices for the high dimensional problem. 
As for the Dirichlet-Laplace prior, we notice that, although its posterior contraction is rate-minimax, its theory requires  $\|\theta^*\|\leq s^{1/2}\log^2(n)$.
We comment that this is actually a very strong condition. Under this condition, even the naive estimator 
$\widehat\theta=0$ can achieve rate-$\{s^{1/2}\log^2(n)\}$ convergence, and minimax rate has merely
 logarithmic improvement.

Given that $\pi_0$ is polynomially decaying, existing literature focuses on how to (adaptively) choose
the global shrinkage $\tau$, i.e., the scale of the prior, such that the order of posterior
contraction rate is (near-)optimal \citep[e.g.,][]{SongL2017, GhoshC2014,VanKV2014,VanSV2017}.
In this work, we will study another aspect of this story, that is how the shape of the prior distribution, i.e., the polynomial order of $\pi_0$, affects the posterior asymptotics.
Our contribution of this work is two-fold. 
First, we show that if the polynomial order of  $\pi_0$ is sufficiently close to 1, we can
achieve Bayesian {\it sharp} minimax, i.e., $r_n/\{2s\log(n/s)\}^{1/2}$ is sufficiently close to 1.
This sharp minimaxity holds for the $L_1$ norm as well.
Our simulation study also demonstrates that it is necessary to choose a tiny polynomial order
in order to obtain the optimal contraction. In Bayesian literature, the sharpness in term of multiplicative constant is barely investigated, and
our work sharpens all existing results on Bayesian posterior convergence for normal means problem.
%It is worth to note that \citep{jiang2009general,brown2009nonparametric} established sharp minimaxity for the Bayes estimator obtained by empirical Bayes approaches,
%but our result is the distributional minimax convergence of the posterior, hence is much stronger than them.
To attain such sharp minimaxity, the choice of $\tau$ will depend on true sparsity ratio $(s/n)$ which in practice is unknown.
Therefore, our second contribution is to propose a Beta modeling on $\tau$. This leads to
a Bayesian sharply minimax inference procedure that is adaptive to unknown sparsity.
Simulations show that this adaptive Beta modeling has an excellent performance.

This paper is organized as follows.
In Section \ref{secmain}, we study the relationship between the polynomial order of $\pi_0$ and
the posterior contraction rate, and establish the Bayesian sharp minimax.
In Section \ref{secada}, we propose an adaptive modeling which doesn't depend on $s$.
Two simulation studies and one real cancer data application are presented in Section \ref{simu}. 
In the end, Section \ref{end} provides more discussions and remarks.
Technical proofs are provided in the Appendix.

Throughout this work, we use $D_n$ to denote the observed data (i.e., $D_n=\{y_i\}_{i=1}^n$), and $\pi(\cdot|D_n)$
to denote the posterior based on data $D_n$. Given two positive sequences $\{a_n\}$ and $\{b_n\}$, 
$a_n\succ b_n$ means $\lim(a_n/b_n)=\infty$, $a_n\asymp b_n$ means $0<\lim\inf(a_n/b_n)$ $\leq\lim\sup(a_n/b_n)<\infty$, and $a_n\sim b_n$ means $\lim(a_n/b_n)=1$.
$\|\cdot\|$ and $\|\cdot\|_1$ denote $L_2$ and $L_1$ norms of a vector respectively.

\section{Sharp Bayesian minimaxity}\label{secmain}
As discussed in the Section \ref{intro}, 
we consider Bayesian inferences for the sparse normal means model under a general prior specification (\ref{gprior}), where $\pi_0$ has a polynomial tail; in other words, we assume the following conditions on the model sparsity and prior distribution $\pi_0$:

[C.1] The true model is sparse $s=o(n)$.

[C.2] The prior density $\pi_0(\cdot)$ is strictly decreasing on $(0,\infty)$ and increasing on $(-\infty, 0)$.

[C.3] The tail of $\pi_0(\cdot)$ is polynomially decaying with polynomial order $\alpha>1$, i.e, there exist some positive constants $M$ and $C_2>C_1$
such that for any $|\theta|>M$, $ C_1|\theta|^{-\alpha}\leq\pi_0(\theta)\leq C_2|\theta|^{-\alpha}$.

Note the condition C.3 (i.e., the polynomial decaying of $\pi_0$) implies the polynomial decaying of $\pi(\theta_i|\tau)$ as well: if $|\theta_i|\gg \tau$, $\pi(\theta_i|\tau)\asymp |\theta_i|^{-\alpha}\tau^{\alpha-1}$. 

For the simplicity of analysis, in this section we only investigate the posterior asymptotics when the global shrinkage parameter $\tau$ is a deterministic value,
under which the posteriors of $\theta_i$'s are mutually independent.

Let's first intuitively understand the posterior properties induced by a polynomial decaying prior. 
The posterior distributions of all $\theta_i$'s independently 
follow
\[
\pi(\theta_i|y_i) =C \exp\{-(y_i-\theta_i)^2/2\}\pi(\theta_i),
\]
for some normalizing constant $C$. Since both functions $\exp\{-(y_i-\theta_i)^2/2\}$ and prior $\pi(\theta_i)$ are unimodal, 
heuristically, the posterior $\pi(\theta_i|y_i)$ will have two major modes, around 0 and $y_i$
respectively.
The posterior mass of the mode around 0 is approximately
$\pi\{\theta_i\in(-\delta,\delta)|y_i\}\approx C \exp\{-y_i^2/2\}\pi(\theta_i\in(-\delta,\delta)) \approx C \exp\{-y_i^2/2\}$
for some $\delta$ satisfying $\delta\succ \tau$;
The posterior mass of the mode around $y_i$ is approximately
$\pi(\theta_i\in y_i\pm \{\delta' \log (n/s)\}^{1/2}|y_i)\approx C \int_{y_i\pm \sqrt{\delta' \log (n/s)}}\exp\{-(y_i-\theta_i)^2/2\}d\theta_i\pi(y_i)\approx C\sqrt{2\pi} \pi(y_i) $ for any small constant $\delta'$.
Therefore, if $\exp\{-y_i^2/2\}\succ \pi(y_i)$, the dominating posterior mode is the one at 0; if $\exp\{-y_i^2/2\}\prec \pi(y_i)$, the dominating posterior  mode is the one at $y_i$.
Note that in the above comparison, we consider different neighbor radiuses around 0 and $y_i$, this is due to the fact that the landscape of $\pi(\theta_i|y_i)$ has two unequally wide modes (refer to Figure \ref{asy} for an illustration).
This heuristic comparison demonstrates a hard thresholding phenomenon for the Bayesian posterior center: when $|y_i|\ll t$, the posterior of $\pi(\theta_i|y_i)$ shrinks to 0; when $|y_i|\gg t$, the 
posterior will be concentrated around $y_i$, where the threshold value $t$ satisfies $\exp\{-t^2/2\}\asymp \pi(t)\asymp t^{-\alpha}\tau^{\alpha-1}$.
This is analogous to the hard thresholding estimator $\hat\theta_i = y_i1(|y_i|\geq t)$ (refer to Figure \ref{sp} for more details).
It is known that the hard thresholding estimator achieves asymptotic sharp minimaxity when $t=\{2\log(n/s)\}^{1/2}$, thus we conjecture that such
sharp minimaxity carries over to the posterior {mean} of Bayesian hard thresholding if the same threshold value is used, i.e., the prior satisfies $\exp\{-\log(n/s)\} \asymp\tau^{\alpha-1}\{2\log(n/s)\}^{-\alpha/2}$.
In addition, to derive minimax posterior contraction beyond minimax posterior mean, we also need to control
the posterior variation, especially the posterior variation
for these zero $\theta_i$'s. Hence, we need to impose sufficiently strong 
shrinkage effect such that the posteriors of these zero $\theta_i$'s contract inside a small
neighborhood around zero. Equivalently, this requires a sufficiently small global shrinkage parameter $\tau$. 
Rigorously, we establish the following theorem whose proof is presented in Appendix \ref{proof1}.

\begin{theorem}\label{main}
	Given a positive constant $\omega$, if $\tau^{\alpha-1}\geq (s/n)^c\{\log(n/s)\}^{1/2}$ for some $c\in(0, 1+\omega/2)$, and 
	$\tau^{\alpha-1}\prec \{(s/n)\log(n/s)\}^{\alpha}$, then
	\begin{equation}\label{mainresult0}
	%\lim E^*(\pi[\|\theta-\theta^*\|\geq \{(2+\omega)^{1/2}+\omega^{1/2}\}\{s\log(n/s)\}^{1/2}|D_n])=0.
	\lim E^*(\pi[\|\theta-\theta^*\|\geq C_1(\omega)\{s\log(n/s)\}^{1/2}|D_n])=0,
	\end{equation}
	where $C_1(\omega)=\sqrt{2+\omega}+\sqrt{\omega}$ and it satisfies
	$\lim_{\omega\downarrow 0}C_1(\omega)=\sqrt{2}$. If furthermore, $\tau^{\alpha-1}\prec(s/n)^\alpha\{\log(n/s)\}^{(\alpha+1)/2}$,
	then
	\begin{equation}\label{mainresult1}
	%\lim E^*(\pi[\|\theta-\theta^*\|_1\geq s\{(2+\omega)^{1/2}+(\omega^2/5)^{1/2}+(\omega/5)^{1/2}\}\{\log(n/s)\}^{1/2}|D_n])=0.
	\lim E^*(\pi[\|\theta-\theta^*\|_1\geq sC_2(\omega)\{\log(n/s)\}^{1/2}|D_n])=0.
	\end{equation}
	where $C_2(\omega)=\sqrt{2+\omega}+\sqrt{w^2/5}+\sqrt{\omega/5}$, and it satisfies 	$\lim_{\omega\downarrow 0}C_2(\omega)=\sqrt{2}$.
\end{theorem}

We have several comments on this theorem. 
There are two conditions for global shrinkage parameter $\tau$. The first condition says that $\tau$ shall not be too small. 
An overly small $\tau$ may cause over-shrinkage for the true nonzero $\theta_i$'s.
This condition also echoes our heuristic argument that $s/n\asymp\tau^{\alpha-1}(2\log(n/s))^{-\alpha/2}$ in the previous paragraph.
The second condition says that $\tau$ shall not be too large, since an overly large $\tau$ fails to impose sufficient shrinkage effect on the zero $\theta_i$'s.
To satisfy both conditions of $\tau$, we need to choose $\alpha\in(1,1+\omega/2)$. 
The theorem provides an $L_1$ contraction result as well. Convergence in $L_1$, comparing with $L_2$ convergence, requires stronger posterior variation control for the zero $\theta_i$'s. Note that $L_1$ contraction result can not be easily derived from the quantification of the  posterior mean bias and posterior variance (which is the proof technique adopted by \citep{VanKV2014,bai2017inverse}). Our proof technique, which directly studies the entry-wise posterior contraction for all the zero $\theta_i$'s, enables us to perform $L_1$ contraction analysis.
There are several insights obtained from the theoretical results of this theorem. First, for any polynomially decaying prior, with proper choice of global shrinkage $\tau$, the order of Bayesian contraction
rate is exactly $O(\{s\log(n/s)\}^{1/2})$.
Note that \citep{VanKV2014, GhoshC2014,van2016conditions} only showed that the order of contraction is $O(M_n\{s\log(n/s)\}^{1/2})$ with $M_n\rightarrow \infty$, since their results
are based on Markov’s inequality. In contrast, our proof is based on constructing hypothesis testing that can test the true parameter versus balls of alternatives with exponentially small error probabilities.
 Secondly, the multiplicative constant of the Bayesian contraction rate is positively related to the polynomial order of $\pi_0(\cdot)$. 
Therefore, to obtain sharply $L_2$/$L_1$ minimax contraction\footnote{The minimax $L_1$ contraction rate is $s\sqrt{(2+o(1))\log(n/s)}$ \citep{DonohoJ1994,zhang2012minimax}}, i.e. the $\omega$ can be sufficiently small, a sufficient choice is to let the polynomial order of $\pi_0$ be very close to 1.
%An important consequence of using a prior whose polynomial order that sufficiently close to 1,
%is that the posterior will also have a very heavy tail. This makes the posterior mean estimator
%unstable

Since the logarithmic term $\log(n/s)$ is asymptotically dominated by $(n/s)^{c}$ for any small constant $c>0$, the conditions of $\tau$ in Theorem \ref{main} can be simplified (by replacing $\log(n/s)$ with arbitrarily small exponent of $(n/s)$), and we obtain the following corollary.
\begin{corollary}\label{coro}
	If $\alpha\leq 1+\omega/2$ and $\tau^{\alpha-1}\asymp (s/n)^c$ for some $c\in[\alpha, 1+\omega/2)$,  then (\ref{mainresult0}) and (\ref{mainresult1}) hold.
\end{corollary}

Theorem \ref{main} and Corollary \ref{coro} suggest an optimal choice for the global shrinkage as $\tau \asymp (s/n)^{(\alpha+\delta)/(\alpha-1)}$ for some non-negative small value $\delta$. 
In practice, when true sparsity is always unknown, this theoretical suggestion is not useful.
Therefore in Section \ref{secada}, we will discuss a full Bayesian approach which is adaptive to unknown sparsity. 
To end this section, we consider a possible alternative choice as  $\tau\asymp (1/n)^{(\alpha+\delta)/(\alpha-1)}$,
that is to substitute the unknown $s$ with 1.
Obviously, if we assume that the sparsity $s$ is a fixed quantity, this choice of $\tau$ is of the same order 
of the optimal suggestion. But if $s$ is increasing with $n$, this leads to suboptimal upper bound for the contraction rate.
The detail is provided in the next theorem.

\begin{theorem}\label{suboptimal}
	If $\tau^{\alpha-1}\geq (1/n)^c\{\log(n/s)\}^{1/2}$ for some $c\in(0, 1+\omega/2)$, and 
	$\tau^{\alpha-1}\prec(s/n)^\alpha\{\log(n/s)\}^{(\alpha+1)/2}$, then
	\begin{equation}\begin{split}\label{mainresult2}
	%&	\lim E^*(\pi[\|\theta-\theta^*\|\geq \{(2+\omega)^{1/2}+\omega^{1/2}\}\{s\log(n)\}^{1/2}|D_n])=0,\\
	%&\lim E^*(\pi[\|\theta-\theta^*\|_1\geq s\{(2+\omega)^{1/2}+(\omega^2/5)^{1/2}+(\omega/5)^{1/2}\}\{\log(n)\}^{1/2}|D_n])=0.\\
	&	\lim E^*(\pi[\|\theta-\theta^*\|\geq C_1(\omega)\{s\log(n)\}^{1/2}|D_n])=0,\\
	&\lim E^*(\pi[\|\theta-\theta^*\|_1\geq sC_2(\omega)\{\log(n)\}^{1/2}|D_n])=0,\\
	\end{split}
	\end{equation}
	for the same functions $C_1(\omega)$ and $C_2(\omega)$ used in Theorem \ref{main}.
\end{theorem}

The proof of this theorem is similar to the proof of Theorem \ref{main}, hence is omitted in this manuscript. 
This result allows one to choose $\tau$ to be of order $(1/n)^{(\alpha+\delta)/(\alpha-1)}$ for any constant $\delta>0$, 
and the resultant $L_2$ contraction rate is of order $\{s\log(n)\}^{1/2}$. 
The rate $\{s\log(n)\}^{1/2}$ is considered to be suboptimal in the literature.
When comparing $C_1(\omega)\{s\log(n)\}^{1/2}$ versus $C_1(\omega)\{s\log(n/s)\}^{1/2}$, we have that (1) the two rates are asymptotically the same when $\log s\prec \log n$ (i.e., the sparsity $s$ increases slower than polynomial
rate); (2) the former one is of the same order with the latter, but has a large multiplicative constant, when $s\asymp n^c$ for some $c\in(0,1)$;
(3) the former one is of strictly greater order, when $\log s\sim \log n$ (e.g., $s=n/\log n$). Note here we only compare the {\em upper bound} of posterior contraction rates obtained by Theorems \ref{main} and \ref{suboptimal}, thus it is not rigorous to claim that the prior specification in Theorem \ref{suboptimal} leads to suboptimal posterior convergence.

\section{Adaptive Bayesian inference}\label{secada}

In this section, we now consider that the information of $s$ is not available, 
hence adaptive ways to determine the global shrinkage $\tau$ are necessary.
In Bayesian paradigm, there are at least two popular approaches to handle the hyperparameters: 
one is the empirical Bayes, i.e., to maximize the marginal likelihood of data, 
and the other one is the full Bayesian approach, i.e., to assign a 
prior distribution on the hyperparameters. For example,
\citet{VanSV2017} studied both approaches for the global shrinkage parameter in the horseshoe modeling, 
and established an adaptive horseshoe Bayesian
inference with a suboptimal contraction rate of order $O(\{s\log n\}^{1/2})$.
In this work, we are particularly interested in constructing an appropriate prior
for $\tau$.

Theorem \ref{main} suggests that $\tau$ decreases to zero if $\tau$ is deterministic. 
This motivates us to design the prior $\pi(\tau)$ to be stochastically decreasing as $n$ increases.
More specifically, the distribution of prior $\pi(\tau)$ shrinks toward 0 under proper rate. In the meantime, this prior must
not shrink too fast, such that it still assigns minimal prior density around the optimal choice $\tau\asymp (s/n)^{(\alpha+\delta)/(\alpha-1)}$.
Our next theorem provides a sufficient condition for the prior $\pi(\tau)$, such that
the Bayesian sharp minimaxity still holds.

\begin{theorem}\label{adaptive}
	If $\alpha\leq 1+\omega/2$, $\pi(\tau)$ satisfies that $-\log\pi\{(s/n)^{(1+\omega/2)/(\alpha-1)}\leq \tau\leq (s/n)^{\alpha/(\alpha-1)}\}\prec s\log(n/s)$ and $-\log\pi\{\tau\geq(s/n)^{\alpha/(\alpha-1)}\}\succ s\log(n/s)$,
	and $\max_j|\theta^*_j|\leq (n/s)^{\omega/(5\alpha)}$, then
	(\ref{mainresult0}) and (\ref{mainresult1}) still hold.
\end{theorem}

The above theorem claims that under proper choice of $\pi(\tau)$, sharp minimaxity is still attainable when we choose $\alpha$ to be close to 1.
Let us first discuss the two conditions on the hyper-prior $\pi(\tau)$. The first condition requires that
$\pi(\tau)$ maintains a minimal prior probability on the optimal range $[(s/n)^{(1+\omega/2)/(\alpha-1)}, (s/n)^{\alpha/(\alpha-1)}]$.
The second condition requires that $\pi\{\tau>(s/n)^{\alpha/(\alpha-1)}\}$ rapidly decays to 0, i.e., the distribution 
$\pi(\tau)$ becomes more and more concentrated around 0.
A particular choice that satisfies these two conditions is that 
\begin{equation}\label{beta}
\tau=\tau_0^{c},\quad\tau_0\sim \mbox{Beta}(1, n)
\end{equation}
for some $c\in(\alpha/(\alpha-1), (1+\omega/2)/(\alpha-1))$.
One can easily verify that
\[
\pi\{\tau\geq(s/n)^{\alpha/(\alpha-1)}\}
=(1-(s/n)^{\alpha/(\alpha-1)/c})^{n}\sim \exp\{-s(n/s)^\delta\}, \mbox{ and}
\]
\[
\pi\{(s/n)^{(1+\omega/2)/(\alpha-1)}\leq \tau\leq (s/n)^{\alpha/(\alpha-1)}\}
\sim \exp\{-s(n/s)^{-\delta'}\}
\]
where $\delta=1-\alpha/(\alpha-1)/c$ and $\delta'=(1+\omega/2)/(\alpha-1)/c-1$
are two small positive constants.
It is worth mentioning that the Beta prior is widely used as a hyperprior in spike-and-slab 
Bayesian modeling \citep{Castillov2012,Rockova2015}. For example, a common spike-and-slab modeling assigns a prior probability $p$ for each
$\theta_i$ being selected into the model, i.e. $\pi(\theta_i\neq0)=p$, and \citet{Castillov2012}
suggested a hyperprior $p\sim \mbox{Beta}(1,4n+1)$.
In the literature, \citet{VanSV2017} proposed a hyper truncated half Cauchy prior for the global shrinkage parameter in the horseshoe modeling, that is
$\pi(\tau)\propto 1(\tau\in[1/n,1])/(1+\tau^2)$.
Both Beta modeling (\ref{beta}) and truncated half Cauchy prior have a compact support within $[0,1]$, but a
big difference between these two is that the Beta prior distribution converges to a Dirac measure at 0 as 
$n$ goes to infinity, but the truncated half Cauchy prior converges to a non-degenerated distribution which is the half Cauchy distribution
truncated within [0,1].

In the above theorem, we also impose an additional technical condition on the magnitude of the true nonzero
$\theta_j$ such that $\log(\max|\theta_j^*|)/(\log(n/s))\leq \omega/5\alpha$. 
Note that $\log(n/s)\rightarrow\infty$, hence this condition still allows that the true signal strength to grow.
If the true signal grows sub-polynomially fast, i.e., $\log(\max|\theta_j^*|)=o(\log(n/s))$, then $\omega$ can be arbitrarily small and we obtain sharply minimax contraction; 
If the true signal grows polynomially fast, i.e., $\max|\theta_j^*|\asymp (n/s)^a$ for some constant $a>0$, then
Theorem \ref{adaptive} still ensures that the rate of contraction is $O(\{s\log(n/s)\}^{1/2})$, despite a larger multiplicative constant.
This constraint on $|\theta_j^*|$ essentially is equivalent to that the prior density on true parameter (i.e., $\pi(\theta^*)$)
is bounded away from 0. Similar conditions, which require that the prior is ``thick'' around true parameter value $\theta^*$,
are regularly used in Bayesian theoretical literature \citep[e.g.,][]{Jiang2007, KleijnV2006, 
	GhosalGV2000, GhosalV2007}.
As mentioned in Section \ref{intro}, Dirichlet-Laplace prior \citep{BhattacharyaPPD2015} also imposes
a upper bound constraint on the magnitude of $\|\theta^*\|$,
but our condition is much weaker.

It is worth to mention that this additional condition is mainly due to the fact that the posterior
distributions among $\theta_i$'s are no longer independent when $\tau$ is subject to a prior distribution.
This condition is only sufficient, and is not appealing to theoreticians. Theorem 3.7 in \citep{VanSV2017}
showed that the Bayesian horseshoe with truncated half Cauchy prior on $\tau$ is capable to achieve
order-$(s\log n)^{1/2}$ contraction rate without any assumption on $|\theta^*_j|$.
Similar result can be derived here if one is only interested in a suboptimal upper bound for the posterior contraction rate:
\begin{theorem}\label{adaptive2}
	If $\alpha\leq 1+\omega/2$, the prior of $\tau$ has support on $[(1/n)^{c/(\alpha-1)},\infty)$ for some $c<1+\omega/2$, and the $\pi(\tau)$ also satisfies that $-\log\pi\{(s/n)^{(1+\omega/2)/(\alpha-1)}\leq \tau\leq (s/n)^{\alpha/(\alpha-1)}\}\prec s\log(n/s)$ and $-\log\pi\{\tau\geq(s/n)^{\alpha/(\alpha-1)}\}\succ s\log(n/s)$, then
	(\ref{mainresult2}) holds.
\end{theorem}

The proof of this theorem is a combination of the proof of Theorems \ref{main} and \ref{adaptive}, hence is omitted.
To construct a $\pi(\tau)$ that satisfies the conditions in this theorem, we can simply modify the above Beta modeling
as: $\tau=\tau_0^{c}$ with $c\in(\alpha/(\alpha-1), (1+\omega/2)/(\alpha-1))$ and $\pi(\tau_0)\propto B(\tau_0; 1, n)1(\tau_0\in[1/n,1])$, where $B(\cdot; a,b)$ is
the density of a Beta distribution.

\section{Simulation and data anlaysis}\label{simu}
This section, we will demonstrate two simulation studies to justify our theoretical discoveries,
as well as a cancer data application.
In the first simulation, we assume that the sparsity $s$ is known in advance, and we empirically compare 
difference choices of the polynomial order of the prior tail.
The theorems presented in Section \ref{secmain} assert that it is sufficient to assign a very small $\alpha$; 
and we would like to use numerical studies to evaluate how good is such a choice, and how necessary is this small-$\alpha$ condition.
In the second simulation, we study the performance of the adaptive prior proposed by
Theorem \ref{adaptive} when $s$ is unknown, and compare it to the adaptive horseshoe prior proposed by \citep{VanSV2017}.
We will also present a prostate cancer real data Bayesian analysis.
\vskip 0.1in
{ \noindent \it Simulation I: Comparison between different choices of polynomial order}
\vskip 0.1in
In this simulation, to study the asymptotic behavior, we let the data dimension increases as $n=50$, 100, 500, 1000 and the sparsity $s$ equals to the rounded value of ${n}^{1/2}$.
The nonzero coefficients are chosen to be $\theta^*_i=\{t\log(n/s)\}^{1/2}$ for $1\leq i\leq s$, where $t=1.2, 2.2, 4.2$ and 6.2. These different choices of $t$ represent
a range of levels of signal strength. To implement a class of polynomial priors with difference tail orders, we let $\pi_0$ be the $t$ distribution with degree of freedom $\alpha-1$, i.e.
\begin{equation}\label{tprior}
\theta_i\sim \mbox{N}(0,\lambda_i^2\tau^2);\,
\lambda_i^2\sim \mbox{IG}((\alpha-1)/2, (\alpha-1)/2),
\end{equation}
and $\tau$ follows a deterministic choice $\tau=(s/n)^c$. This leads to a simple Gibbs update
\begin{equation}\label{gibbs}
\lambda_j^2\sim \mbox{IG}(\frac{\alpha}{2}, \frac{\alpha-1}{2}+\frac{\theta_i^2}{2\tau^2}),\,
\theta_i\sim \mbox{N}((1+\frac{1}{\lambda_i^2\tau^2})^{-1}y_i, (1+\frac{1}{\lambda_i^2\tau^2})^{-1}).
\end{equation}
We consider six different choices of prior modeling:
(1) $\alpha=1.1$, $c=(\alpha+0.05)/(\alpha-1)$;
(2) $\alpha=2.1$, $c=(\alpha+0.05)/(\alpha-1)$;
(3) $\alpha=3.1$, $c=(\alpha+0.05)/(\alpha-1)$;
(4) $\alpha=2.1$, $c=1/(\alpha-1)$;
(5) $\alpha=3.1$, $c=1/(\alpha-1)$;
(6) horseshoe prior with global shrinkage $\tau=(s/n)\{\log(n/s)\}^{1/2}$.
In the above five $t$-prior specifications, there are two choices for global shrinkage $\tau$. One is 
$(s/n)^{(\alpha+0.05)/(\alpha-1)}$, which satisfies the upper bound condition on $\tau$ in Theorem \ref{main}.
By our heuristic arguments in Section \ref{secmain}, under such prior specification, $\theta^*_i\approx\{2(\alpha+0.05)\log(n/s)\}^{1/2}$ posts
the most difficult problem. The other choice is $(s/n)^{1/(\alpha-1)}$ which satisfies the lower bound condition on $\tau$ in Theorem \ref{main}.
Note that \citet{GhoshC2014} claimed that an optimal choice for $\tau$ is $[(s/n)\{\log(n/s)\}^{1/2}]^{1/(\alpha-1)}$. The difference
between $(s/n)^{1/(\alpha-1)}$ and $[(s/n)\{\log(n/s)\}^{1/2}]^{1/(\alpha-1)}$ is only a logarithmic 
term. The horseshoe prior has a polynomial tail with order $\alpha=2$, and the choice of its $\tau$ follows the suggestion of \citep{VanKV2014}.
And asymptotically, this horseshoe prior is almost the same to $t$-distribution with $\alpha=2.1$ and $c=1/(\alpha-1)$.
All simulation results are based on the average of 100 replications.

In Figure \ref{order}, we compare the posterior contraction among these 6 priors. We estimate their posterior probability
$\pi(\|\theta-\theta^*\|^2\geq 2.2s\log(n/s)|D_n)$ by posterior samples, and plot this probability with respect to the different
$n$ and $t$ values. For minimax Bayesian procedure, this probability converges to 0 as $n$ increases, regardless of the magnitude of $\theta^*$.
The figure clearly indicates that $\alpha=1.1$ has the best performance.
The posterior probability always decreases toward 0 for all different $t$'s.
For the rest 5 priors, their posteriors don't contract into the $\{2.2s\log(n/s)\}^{1/2}$-neighborhood
for some value of $t$. It is worth to mention that for the $t$ prior with $\alpha=1.1$, the case $t=2.2$
leads to the slowest convergence, as the red curve is decreasing very slowly. This is because,
as we discussed, $t=2.2$ corresponds to the most difficult scenario for $\alpha=1.1$.
Besides, the plots of horseshoe and $t$-distribution with $\alpha=2.1$ and $c=1/(\alpha-1)$ have very
similar patterns, since 
these two prior specifications are almost equivalent in terms of their tail behaviors.

\begin{figure}[htbp]
	\begin{center}
		\includegraphics[width=13cm]{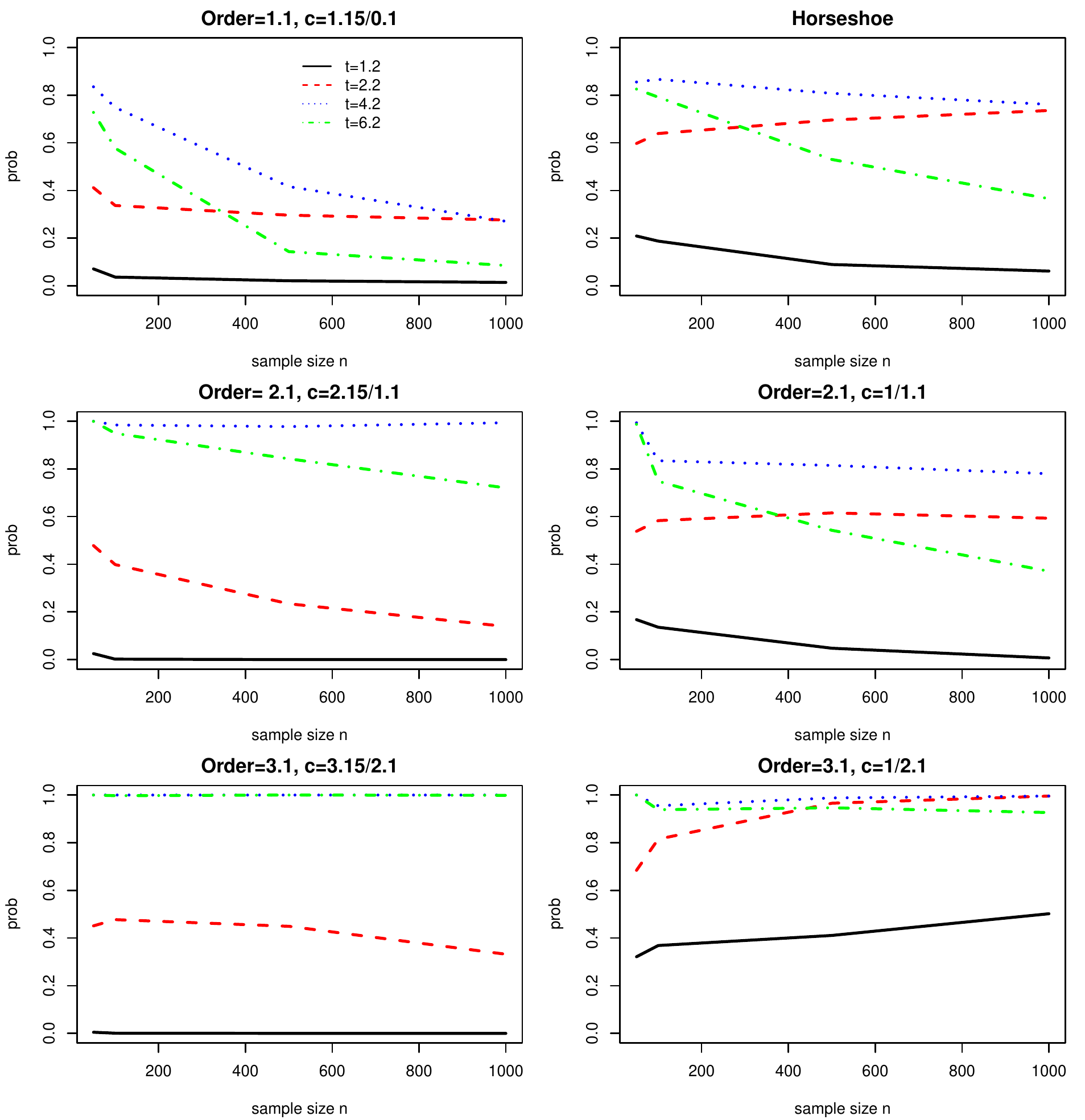}
		\caption{
			Posterior contraction of the 6 prior specifications.
		}\label{order}
	\end{center}
\end{figure}

In Figure \ref{err}, we present the some comparisons of the posterior mean of squared $L_2$ error $E(\|\theta-\theta^*\|^2|D_n)$ and 
posterior mean of $L_1$ error $E(\|\theta-\theta^*\|_1|D_n)$. 
As a reference for the comparison, we also plot the curves corresponding the the minimax squared $L_2$ and $L_1$ errors, namely,
$2s\log(n/s)$ and $s\{2\log(n/s)\}^{1/2}$ respectively.
Note that when $s={n}^{1/2}$, the suboptimal contraction rate $2s\log(n)=2(2s\log(n/s))$ and $s\{2\log n\}^{1/2}=\{2\}^{1/2}s\{2\log (n/s)\}^{1/2}$.
When $t=1.2$, i.e., signals are weak, the $L_2$ errors of all 4 priors don't exceed the minimax rate. 
This is not surprising, because under weak signals, any method that imposes enough shrinkage effect, including the
naive estimator $\widehat\theta=0$,
will induce an $L_2$ error that is smaller than minimax error. It also shows that the $t$-prior with $\alpha=2.1$, $c=(\alpha+0.05)/(\alpha-1)$ does have a
better $L_2$ error than $t$-prior with $\alpha=1.1$ under weak signals. 
When $t=2.2$, i.e.,the signal strength is in the boundary case, horseshoe and $t$-prior with $c=1/(\alpha-1)$ begin to exceed the minimax $L_2$ error, while
the two $t$-priors with $c=(\alpha+0.05)/(\alpha-1)$ have $L_2$ errors that is almost the same as, but slightly higher than, the minimax error.
When $t=4.2$, i.e., signals are strong, $t$-prior with $\alpha=1.1$ is the only one that achieves asymptotic sharp minimaxity.
When $t$ is even larger (which is not presented in the Figure \ref{err}), the $L_2$ error of $t$-prior with $\alpha=1.1$ will be much smaller than the
minimax rate, but the $L_2$ errors for the rest three are much larger than the minimax rate.
In summary, a small polynomial order universally ensures that the $L_2$ estimation error is asymptotically
 bounded by the minimax rate. 
As for the error rates under $L_1$ norm, it is much more sensitive to small variations in the coordinates
than the $L_2$ error. 
The choice $\tau=(s/n)^{(\alpha)+0.05/(\alpha-1)}$ satisfies the upper bound condition on $\tau$ in Theorem \ref{main}, and as we discussed, guarantees sufficient posterior shrinkage for the zero $\theta_i$'s,
hence it leads to small $L_1$ error. This argument is consistent to our simulation results: the two priors with 
choice $c=(\alpha+0.05)/(\alpha-1)$ have much smaller $L_1$ errors. Same to the our comparison of $L_2$ errors, only the small polynomial 
order prior ($\alpha=1.1$) ensures sharp minimaxity.
Similar to Figure \ref{order}, we see that the  $t$-distribution with $\alpha=2.1$ and $c=1/(\alpha-1)$ has almost 
identical performance with the horseshoe prior.

\begin{figure}[htbp]
	\begin{center}
		\includegraphics[width=13cm]{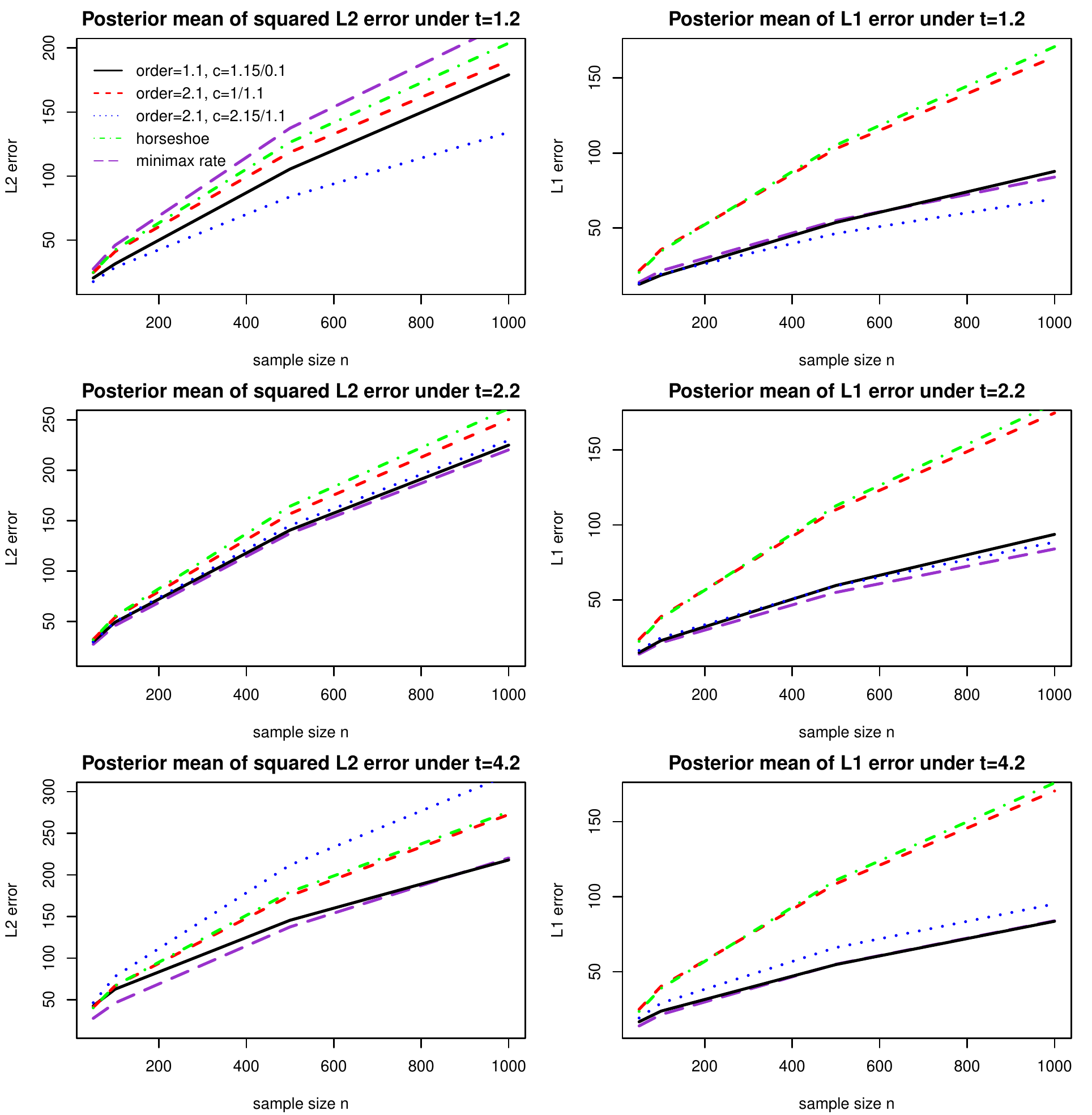}
		\caption{
			Posterior mean sqaured $L_2$ and $L_1$ errors of the 4 selective prior specifications.
		}\label{err}
	\end{center}
\end{figure}

The above simulation results successfully demonstrate the sharpness of the Bayesian 
minimax contraction when the polynomial order $\alpha$ of the prior is close to 1. 
However, there are still some discrepancies between the displayed finite-sample behaviors %of $t$ prior with $\alpha=1.1$
and the Bayesian hard thresholding phenomenon described in Section \ref{secmain}.
%, especially for the case of $t=4.2$.
According to the Bayesian hard thresholding phenomenon, when $|y_i|$ is greater than the thresholding value $\sqrt{2\alpha\log(n/s)}$, its posterior 
will have one dominating mode which is approximately $\mbox{N}(y_i,1)$. This implies that when 
$t=4.2>2\alpha=2.2$, the posterior mean squared $L_2$ error or $L_1$ error of 
those nonzero $\theta_i$'s is asymptotically of order $O(s)$, which thus shall lead to a much smaller estimation error comparing with the boundary case of $t=2.2$.
But in Figure \ref{err}, there is no noticeable difference for posterior mean errors between $t=2.2$ and $t=4.2$ under the $t$ prior with $\alpha=1.1$.
This is because our asymptotic theory relies on the sparsity assumption such that $\log(n/s)\rightarrow\infty$,
while in our simulation experiments, the ratio of $n/s$ is not large enough to reflect the asymptotic behavior.
To illustrate it, Figure \ref{asy} plots the histograms of posterior $\pi(\theta_i|y_i=[3.2\log(n/s)]^{1/2})$
using $t$-prior with $\alpha=1.1$ and $\tau = (s/n)^{1.15/0.1}$, for different values of $n/s$.
As showed, the posterior does have two modes, but the mode around $y_i$ isn't the dominating one
until $n/s$ is as huge as 100,000.
In contrast, other choices of prior with different polynomial order of tail decaying, for example,
the horseshoe prior requires a much smaller value of $n/s$ to have the mode around $y_i$ be dominating.
This implies that in a small-sample real application, prior with polynomial order close to 1 is
less powerful in terms of detecting signals, i.e., yields a sparser model selection result, comparing with
the horseshoe prior.
But on the other hand, horseshoe prior specification fails to induce sufficient shrinkage effect 
for the zero $\theta_i$'s, therefore its estimation performance for the whole vector $\theta$ is
 still inferior.
For $t$-prior with $\alpha=1.1$, we also plot its posterior mean shrinkage coefficient $E(\theta_i|y_i)/y_i$
with respect to values of $c=y_i^2/\log(n/s)$ and $n/s$ in Figure \ref{sp}.
As $n/s$ tends to infinity, the posterior shrinkage does behave more and more similarly to the hard thresholding.

\begin{figure}[htbp]
	\begin{center}
		\includegraphics[width=10cm]{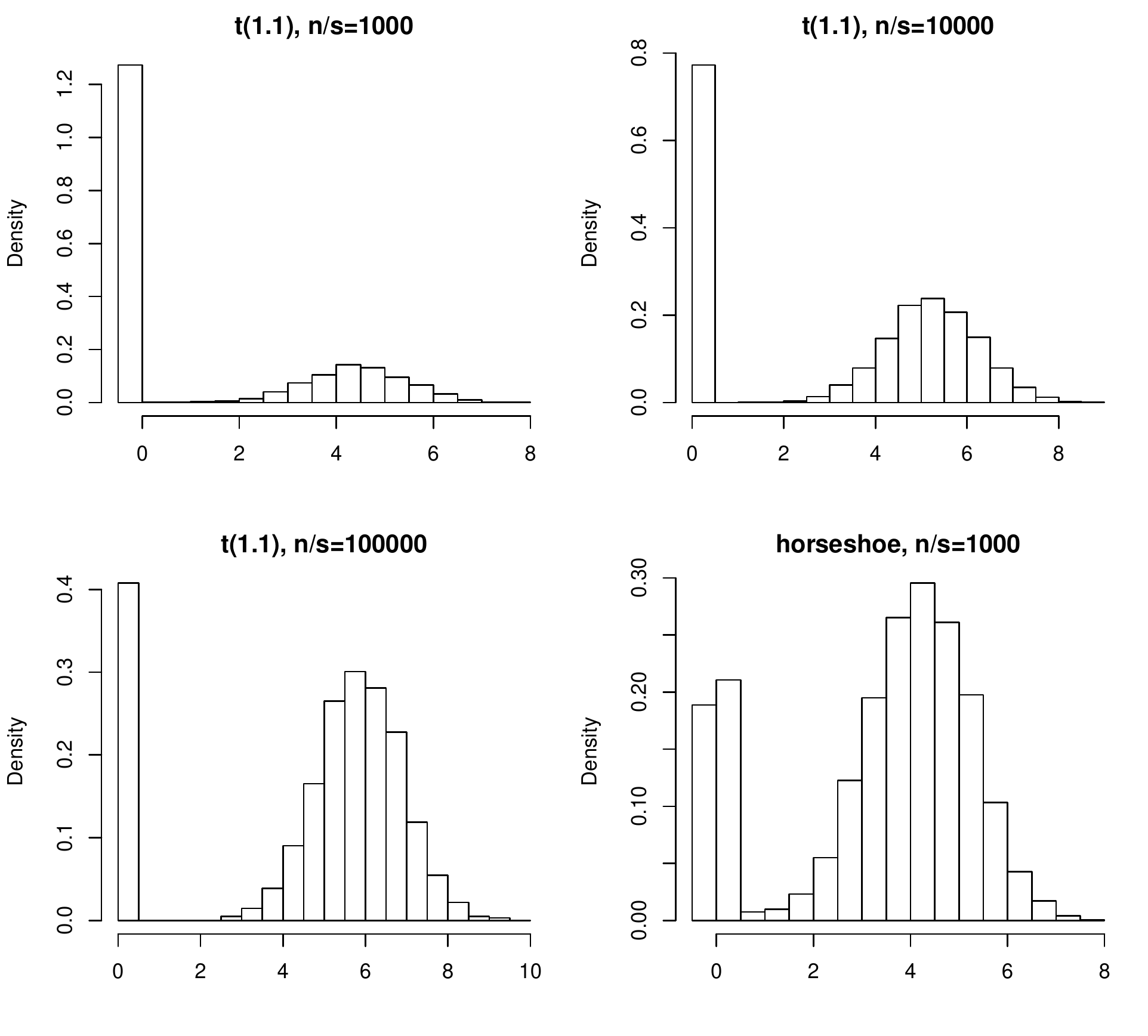}
		\caption{
			Histograms for the posterior $\pi(\theta_i|y_i=(3.2\log(n/s))^{1/2})$.
		}\label{hist}\label{asy}
	\end{center}
\end{figure}

\begin{figure}[htbp]
	\begin{center}
		\includegraphics[width=7cm]{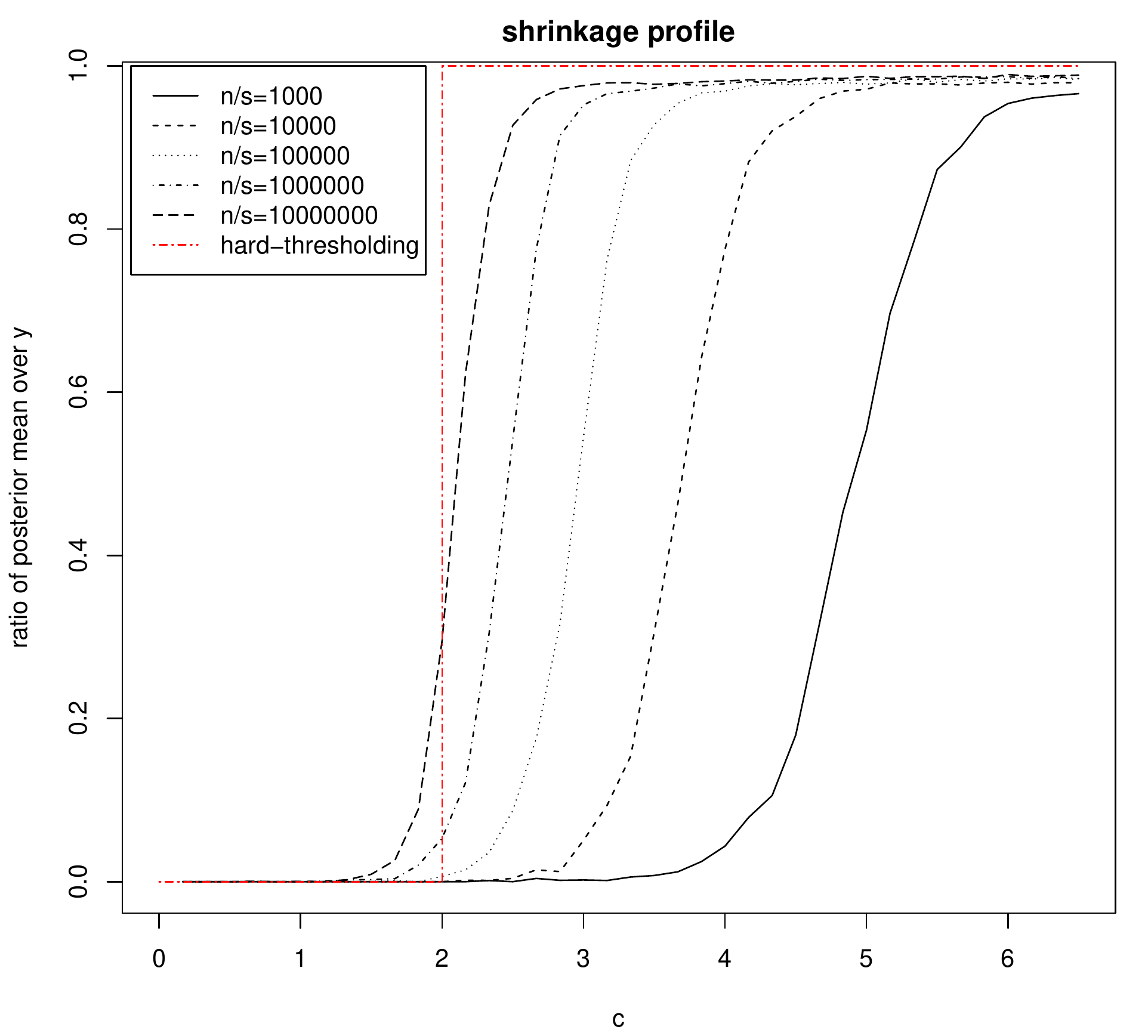}
		\caption{
			Psterior shrinkage profile $E(\theta_i|y_i)/y_i$ under $t$-prior with $\alpha=1.1$ and $\tau = (s/n)^{1.15/0.1}$.
		}\label{sp}
	\end{center}
\end{figure}

In additional, several simulation experiments are presented in the Supplementary Material, where we explore (1) the posterior contraction of shrinkage prior under varying nonzero $\theta$'s scenario, (2) the posterior convergence performances for nonzero $\theta$'s and zero $\theta$'s respectively,
and (3) the uncertainty quantification and Bayesian model selection of shrinkage priors.

\vskip 0.1in
{ \noindent \it Simulation II: Adaptive Bayesian modeling and comparisons}
\vskip 0.1in
In the second simulation, we no longer assume that $s$ is known, and now the global shrinkage $\tau$ is chosen in an adaptive Bayesian manner.
We compare the following two adaptive Bayesian procedures.
The first prior is constructed based on Theorem \ref{adaptive}. We consider a $t$-prior with $\alpha=1.1$, and $\tau$ follows the Beta modeling (\ref{beta})
with $c=(\alpha+0.05)/(\alpha-1)$.
As for the posterior sampling of this adaptive $t$-prior, in addition to (\ref{gibbs}), 
the marginal condition distribution of $\tau_0$ is
\[
\pi(\tau_0|\mbox{rest})\propto \frac{1}{\tau_0^{cn}}\exp\left\{-\frac{1}{\tau_0^{2c}}\sum_i\frac{\theta_i^2}{2\lambda_i^2}\right\}\frac{(1-\tau_0)^{n-1}}{n+1}1(\tau_0\in(0,1)).
\]
Since this conditional posterior has a compact support, it can be sampled via the inverse cumulative-distribution sampling.
Another prior is the horseshoe prior with $\tau$ following a half-Cauchy distribution truncated on $[1/n,1]$ \citep{VanSV2017}.
The simulation settings for data dimension and signal strength are exactly the same as the first
simulation study.
Obviously, one shall expect that the performances of adaptive Bayesian approaches are worse than 
the case that $s$ is known.

Figure \ref{ada} demonstrates a comprehensive comparison between adaptive $t$-prior and adaptive horseshoe prior, in terms
of posterior contraction, posterior mean square $L_2$ error and posterior mean $L_1$ error. 
It shows that the adaptive $t$-prior has a better performance in almost every aspect.
The posterior contraction plots of $t$-prior always have a decreasing trend towards 0 under different 
signal strengths, and its convergence pattern is very similar to Figure
\ref{order}. This implies that the Beta modeling of $\tau$ is a good substitute for the optimal 
choice $\tau=(s/n)^{(\alpha+0.05)/(\alpha-1)}$.
The posterior contraction of the adaptive horseshoe, on the other side, doesn't converge at all. 
The plot pattern is also quite different from the posterior contraction plot in Figure \ref{order},
especially for the case $t=6.2$. This somehow indicates that the truncated half-Cauchy prior doesn't
adapt well to large signals.
For the posterior mean $L_2$ error of the adaptive $t$-prior, when the signal is weak or strong, its
error is well bounded by minimax rate if $n$ is large. When the signal strength is moderate,
its error slightly exceeds the minimax rate. However, there is no trend showing that the
$L_2$ error will increase faster than the minimax rate, thus we believe that if $n$ continues growing and $\alpha$ is closer to 1,
the error of adaptive $t$-prior should be asymptotically bounded by the minimax rate.
But the adaptive horseshoe prior induces a much larger error than the minimax rate, except for the 
weak signal situation. Similarly, in term of the $L_1$ norm, the
adaptive $t$-prior attains the minimax rate, and it 
clearly outperforms the horseshoe prior regardless of the signal strength.

\begin{figure}[htbp]
	\begin{center}
		\includegraphics[width=13cm]{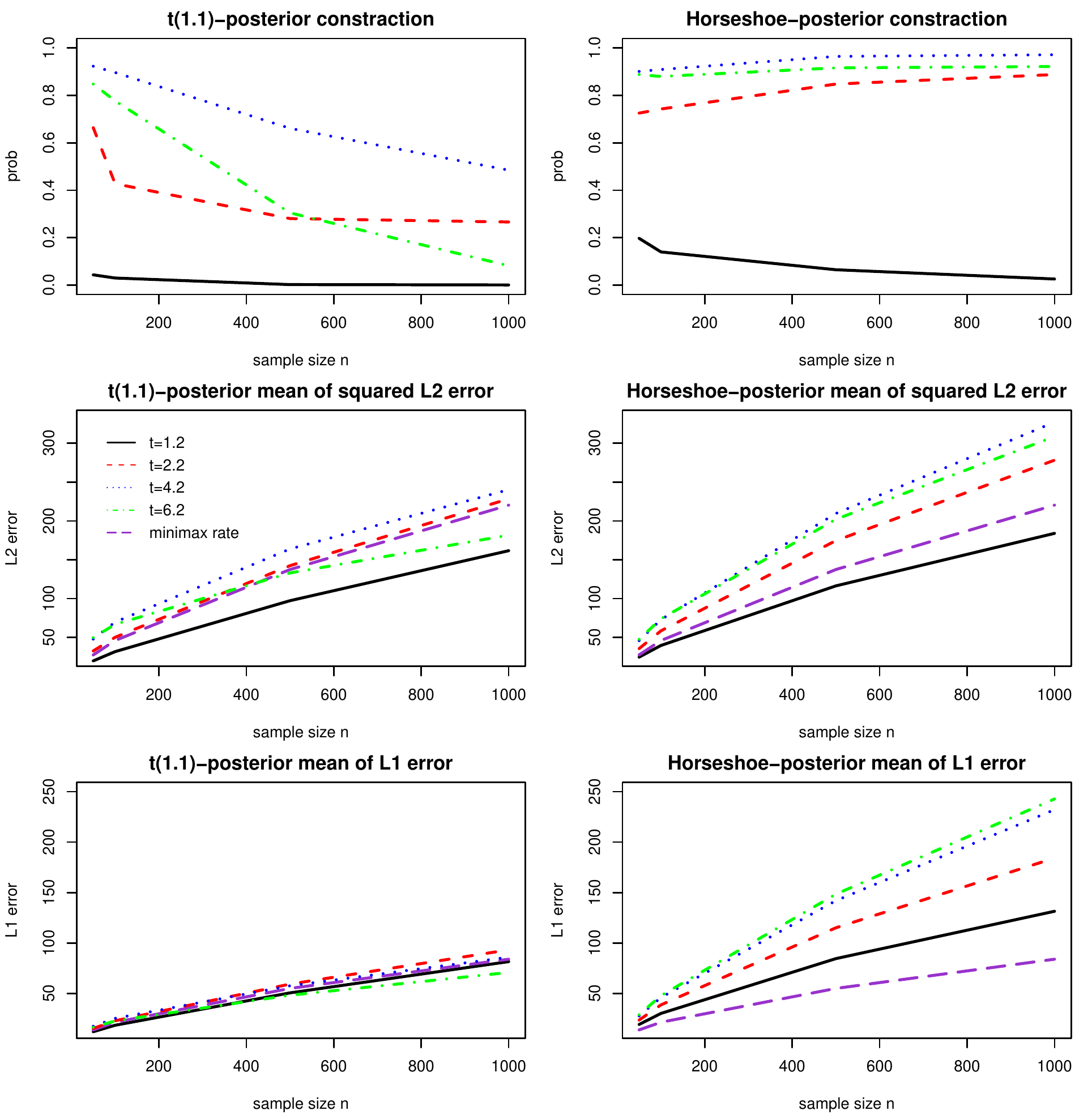}
		\caption{
			Comparison between adaptive $t$-prior and adaptive horseshoe prior.
		}\label{ada}
	\end{center}
\end{figure}

As a conclusion, the presented two simulation studies  demonstrate the necessity
of choosing $\alpha$ to be sufficiently close to 1.
A prior specification as simple as $t$-distribution with a tiny degree of freedom 
ensures supreme Bayesian contraction and estimation.
The proposed adaptive Beta modeling on the global shrinkage $\tau$ leads to a very stable
result and significantly outperforms the adaptive horseshoe prior.

\vskip 0.1in
{ \noindent \it Real data set analysis}
\vskip 0.1in
We consider a popular prostate cancer dataset \citep{efron2008microarrays,singh2002gene} 
from a microarray experiment which consists of expression levels for $n=6033$ genes from 50 normal control subjects and 52 
cancer patients\footnote{The data set is available in the book \citep{efron2012large}.}. 
Two-sample tests are performed to compare the expression level of each gene between control and patient groups.
The corresponding p-values are thereafter converted into $z$-statistics, i.e.,
$z_i=\Phi^{-1}(p_i/2)$ for $i=1,\dots,n$. Hence, it is appropriate to model these $z$-statistics as a normal means model
$z_i=\theta_i+\epsilon_i$, where $\theta_i=0$ if the mean expression levels for the $i$th gene are
the same between control and patient population. We use Bayesian shrinkage to make inference on the
 parameter $\theta$.

We implement the adaptive $t$ prior with $\alpha=1.1$ and $\tau$ following the Beta modeling (\ref{beta}),
and the adaptive horseshoe prior with $\tau$ following truncated half Cauchy prior.
Note that horseshoe prior has already been used to analyze this prostate cancer data
in the literature \citep{BhattacharyaPPD2015,bai2017inverse,bhadra2017horseshoe+},
but all these applications choose $\tau$ to follow the non-truncated half Cauchy prior.
As illustrated by \citep{VanSV2017}, the empirical performance between truncated half Cauchy hyper-prior and 
non-truncated half Cauchy hyper-prior are quite different, hence the horseshoe posterior summary presented 
in this section is not comparable to the results in the literature.

\begin{figure}[htbp]
	\begin{center}
		\includegraphics[width=11cm]{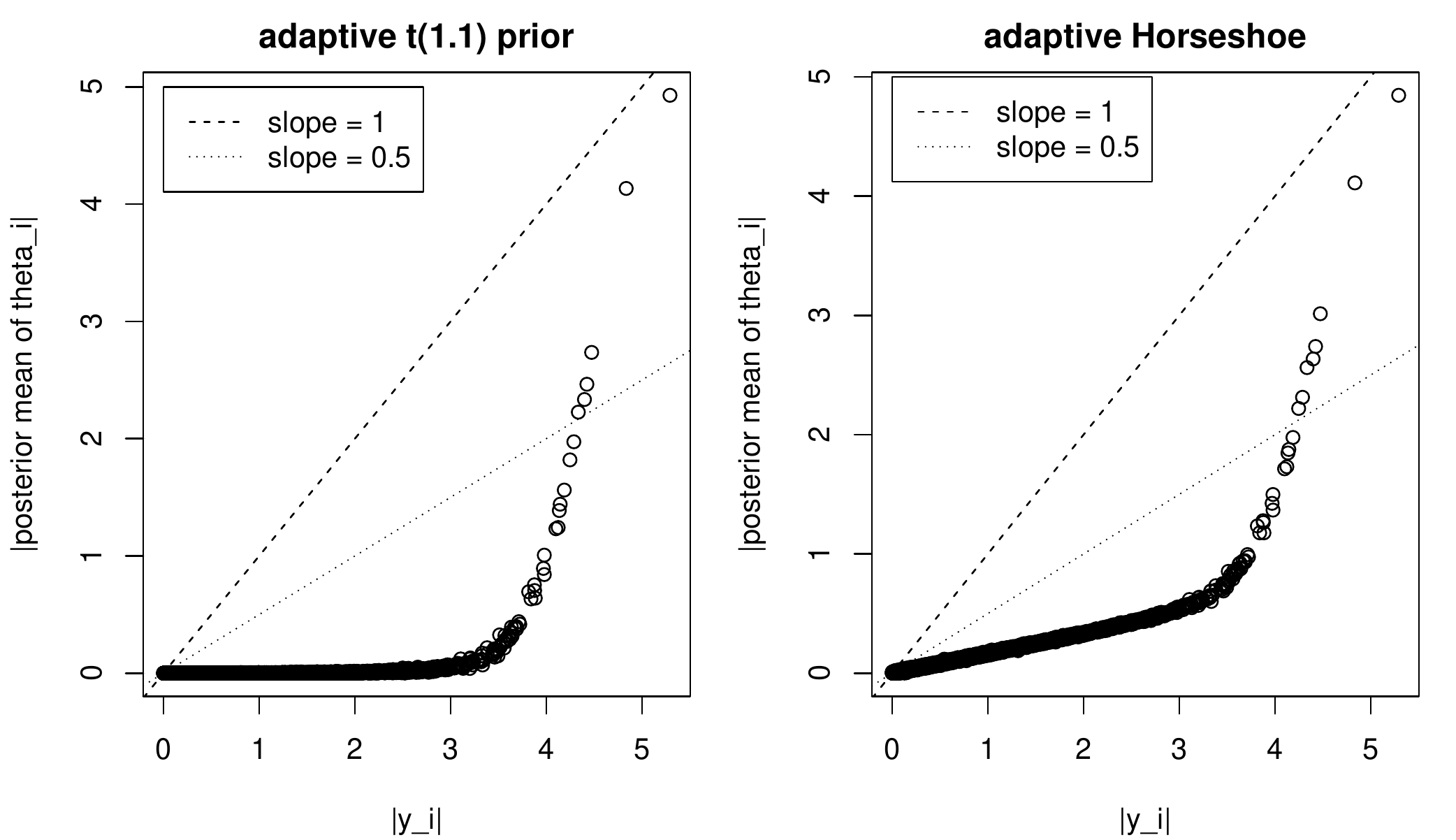}
		\caption{
			Comparison between adaptive $t$-prior and adaptive horseshoe prior.
		}\label{real-comp}
	\end{center}
\end{figure}

In Figure \ref{real-comp}, we plot the posterior means $|E(\theta_i|D_n)|$ against the observations
$|y_i|$. For larger $|y_i|$'s, the posterior means between adaptive $t$-prior and adaptive horseshoe
prior are close. For smaller $|y_i|$'s, the adaptive $t$-prior apparently induces stronger 
shrinkage effect. This observation is consistent to our simulations results that horseshoe prior imposes
insufficient posterior shrinkage for the zero $\theta_i$'s. 
If we perform a variable selection by selecting genes for which  $|E(\theta_i|D_n)|/|y_i|>1/2$, 
then horseshoe selects the top 8 genes and $t$-prior selects the top 6 genes. This is also consistent
to our previous arguments that the $t$-prior with polynomial order close to 1 is less powerful than horseshoe prior.
Note that the posterior of $\pi(\theta_i|D_n)$ has two major modes around $0$ and $y_i$,
hence the selection rule $|E(\theta_i|D_n)|/|y_i|>1/2$ heuristically means that the posterior mass for the 
mode around $y_i$ is greater than half.

\section{Final remarks}\label{end}
In this work, we study the Bayesian inference on high dimensional sparse normal sequence model
with a polynomially decaying prior distribution. Our main result Theorem \ref{main}
reveals the connection between the upper bound of the posterior contraction and the polynomial order $\alpha$.
This provides a sufficient condition to induce sharp posterior minimaxity. That is, choosing a sufficiently tiny $\alpha-1$, the ratio between Bayesian posterior contraction rate and minimax will be 
sufficiently close to 1.
We conjecture necessity holds for Theorem \ref{main} as well, such that
the smaller the $\alpha-1$, the better the Bayesian contraction in terms of the multiplicative constant.
 Empirical studies also show great improvement for the accuracy of Bayesian shrinkage procedure
using a $t$-prior with a tiny degree of freedom. 
Our study considers $\alpha$ to be a fixed, sufficiently small hyperparameter.
Alternatively, one can investigate the choice of letting the polynomial order $\alpha$ decrease to 1 as $n$ increases, i.e.,
$\lim\alpha_n=1$ under proper rate.
Another related question will be: is it a good choice to use an improper prior with exactly $\alpha=1$, e.g.
$\pi_0(\theta)\propto \theta^{-1}$. Our theoretical results break down when $\alpha=1$ (the term $\alpha-1$ appears in the denominator),
and our technical tool which follows the arguments of Le Cam-Birg{\' e} testing theory \citep{Birge1984,Barron1998,LeCam1986} only works
for proper prior specifications.

The primary research interest of this paper is on the $L_2$ and $L_1$ posterior contraction rates. 
Another important research objective is sparsity recovery, i.e., to identify the set $\{j: \theta_j^*\neq 0\}$.
Given a continuous posterior distribution induced by a shrinkage prior, one easy way to perform model selection is 
to do a threshold truncation, that is, a variable is selected if its posterior summary such as posterior mean is greater than some thresholding value. This simply approach has been widely used in the literature, however, it usually 
leads to over-selection with the number of false positives being of order of $O(s)$ \citep[e.g., Theorem 3.4 of][]{BhattacharyaPPD2015}.
%and we believe the same
%selection behavior holds here as well. It is worth to mention that a recent work by \cite{SongC2018} rigorously proves that any monotone rate-minimax estimator will yield false discoveries. 
Another different model selection approach is to select $\theta_i$'s whose marginal credible intervals exclude 0, and the consistency of this Bayesian selection method is investigated by \citep{VanSV2017_2} for the horseshoe prior.

This work focus on the normal sequence model, thus it would be of substantial interest to conduct similar 
investigation for general regression model. Our results heavily rely on the independence among $y_i$'s,
and it is not trivial to extend these results to regression model with correlated design matrix.
\citet{SongL2017} studies the posterior asymptotics for general linear regression model,
including order-$(s\log p/n)^{1/2}$ $L_2$ contraction and model selection consistency, when a polynomially decaying prior 
is used.
We believe that the choice of polynomial order also plays a role for the multiplicative constant of the 
posterior contraction rate under regression model, and we conjecture that the optimal choice of $\alpha$ will depend on the eigen structure of the design matrix.
If the design matrix $X$ is nearly orthogonal, e.g., all entries of $X$ follow independent Gaussian distribution,
we conjecture that the same results as Theorem \ref{main} will still hold, and one need to choose $\alpha\approx 1$
in order to obtain optimal Bayesian contraction.

\appendix
\section{Technical proofs}
\subsection{Proof of Theorem \ref{main}}\label{proof1}
The proof consists of two parts. Since the posteriors of $\theta_i$'s are independent,
in Part I, we study the posterior contraction for the nonzero $\theta_i$'s and in Part II, we 
study the posterior contraction for the zero $\theta_i$'s. 
First, let us state some useful lemmas.

\begin{lemma}[Lemma 1 of \citep{LaurentM2000}]\label{chi}
	Let $\chi^2_{d}(\kappa)$ be a chi-square random variable with degree of freedom $d$ and noncentral parameter $\kappa$, then we have the following concentration inequality
%	\[Pr( \chi^2_{d}(\kappa)>d+\kappa+2x+\sqrt{(4d+8\kappa)x})\leq \exp(-x), \mbox{ and}\] 
	\[Pr( \chi^2_{d}(\kappa)<d+\kappa-\{(4d+8\kappa)x\}^{1/2})\leq \exp(-x),\]
	for any $x>0$.
\end{lemma}

\begin{lemma}[Theorem 1 of \citep{ZubkovS2013}]\label{lemmad}
        Let $X$ be a Binomial random variable $X\sim \mbox{B}(n,v)$. For any $1<k<n-1$
	\[ 
	Pr(X\geq k+1)\leq 1- \Phi(\mbox{sign}(k-nv)\{2nH(v, k/n)\}^{1/2}),
	\]
	where $\Phi$ is the cumulative distribution function of standard Gaussian distribution and
	$H(v, k/n)= (k/n)\log(k/nv)+(1-k/n)\log[(1-k/n)/(1-v)]$.
\end{lemma}

The next lemma is a refined result of Lemma 6 in \citep{Barron1998}:
\begin{lemma}\label{lemmab}
	Let $f^*$ be the true probability density of data generation, $f_{\theta}$ be the likelihood function
	with parameter $\theta\in\Theta$, and $E^*$, $E_\theta$ denote the corresponding expectation
	respectively. Let $B_n$ and $C_n$ be two subsets of the parameter space $\Theta$, and $\phi_n$ be some
	testing function satisfying 
	$\phi_n(D_n)\in[0,1]$ for any realization $D_n$ of the data generation. If $\pi(B_n)\leq b_n$, $E^*\phi(D_n)\leq b_n'$,
	$\sup_{\theta\in C_n}E_{\theta}(1-\phi(D_n))\leq c_n$, and 
	\[P^*\left\{\frac{m(D_n)}{f^*(D_n)}\geq a_n \right\}\geq 1-a_n',\]
	where $m(D_n)=\int_{\Theta} \pi(\theta)f_{\theta}(D_n)d\theta$ is the margin density of $D_n$, then,
	\[
	\begin{split}
	&E^*\left(\pi(C_n\cup B_n)|D_n)\right)\leq \frac{b_n+c_n}{a_n}+a_n'+b_n'.
	\end{split}
	\]
\end{lemma}
\begin{proof}
	Define $\Omega_n$ be the event of $(m(D_n))/(f^*(D_n))\geq a_n$,
	and $m(D_n, C_n\cup B_n) = \int_{C_n\cup B_n} \pi(\theta)f_{\theta}(D_n)d\theta$. Then
	\[
	\begin{split}
	&E^*\pi(C_n\cup B_n)|D_n) = E^*\pi(C_n\cup B_n)|D_n)(1-\phi(D_n))1_{\Omega_n}\\
	+&E^*\pi(C_n\cup B_n)|D_n)(1-\phi(D_n))(1-1_{\Omega_n})+ E^*\pi(C_n\cup B_n)|D_n)\phi(D_n)\\
	\leq& E^*\pi(C_n\cup B_n)|D_n)(1-\phi(D_n))1_{\Omega_n}+E^*(1-1_{\Omega_n})+ E^*\phi(D_n)\\
	\leq& E^*\pi(C_n\cup B_n)|D_n)(1-\phi(D_n))1_{\Omega_n}+b_n'+a_n'\\
	\leq& E^*\{m(D_n, C_n\cup B_n)/a_nf^*(D_n)\}(1-\phi(D_n))+b_n'+a_n'.
	\end{split}
	\]
	
	By Fubini theorem,
	\[\begin{split}
	&E^*(1-\phi(D_n))m(D_n, C_n\cup B_n)/f^*(D_n) \\
	=& \int_{C_n\cup B_n} \int_{\mX}[1-\phi(D_n)]f_\theta(D_n) dD_n\pi(\theta)d\theta\\
	\leq&\int_{C_n} E_{\theta}(1-\phi(D_n))\pi(\theta)d\theta+\int_{ B_n}  \int_{\mX}f_\theta(D_n) dD_n\pi(\theta)d\theta\leq b_n+c_n.
	\end{split}\]
	
	Combining the above inequalities leads to the conclusion.
\end{proof}
{\noindent Part I: posterior contraction rate of nonzero $\theta_i$'s}
\vskip 0.1in
It is equivalent to consider the situation that 
$\tilde y=\vartheta_1+\epsilon$, where $\tilde y,\vartheta_1\in \BR^s$, $\epsilon\sim N(0,I_s)$.
The parameter $\vartheta_1$ is subject to prior $\prod_{j=1}^s\pi(\vartheta_{1,j})$. The parameter 
$\vartheta_1$ hence correspeonds to the subvector of $\theta$: $\vartheta_1=(\theta_j)_{j\in\{\theta_j^*\neq 0\}}$, and its true value is $\vartheta_1^*=(\theta_j^*)_{j\in\{\theta_j^*\neq 0\}}$.
We want to show that 
\begin{equation}\label{part1}
E^*\pi(\|\vartheta_1-\vartheta_1^*\|\geq \{(2+\omega)s\log(n/s)\}^{1/2}|\tilde y)\leq \exp\{-c_0s\log(n/s)\},
\end{equation}
for %some tiny constant $\omega$ and 
some positive constant $c_0$.

Let's consider the testing function $\phi(\tilde y)=1(\|\tilde y-\vartheta_1^*\|\geq \{\delta s\log(n/s)\}^{1/2})$
where $\delta$ is a positive but tiny constant.
This testing function satisfies
\begin{equation}\label{tfun1}
E_{\vartheta_1^*}\phi(\tilde y)= Pr(\chi_s^2\geq \delta s\log(n/s))\leq \exp\left\{-\frac{\delta s\log (n/s)}{2+\delta_0}\right\}
\end{equation}
where $E_{\vartheta_1}$ denotes the expectation over $\tilde y$ with respect to true parameter being $\vartheta_1$, %(hence $E_{\vartheta_1^*}=E^*$)
and the last inequality holds for any fixed $\delta_0>0$ when $n$ is sufficiently large due to
Lemma \ref{chi}. And for any $\vartheta_1\in C_n=\{\vartheta_1\in\BR^s:\|\vartheta_1-\vartheta_1^*\|\geq \{(2+\omega)s\log(n/s)\}^{1/2}\}$,
\begin{equation}\label{tfun2}
\begin{split}
&E_{\vartheta_1}[1-\phi(\tilde y)]= Pr(\|\tilde y-\vartheta_1^*\|\leq \{\delta s\log(n/s)\}^{1/2}|\vartheta_1)\\
\leq& Pr(\|\tilde y-\vartheta_1\|\geq \|\vartheta_1-\vartheta^*_1\|-\{\delta s\log(n/s)\}^{1/2}|\vartheta_1)\\
=& Pr(\chi_s^2\geq[\|\vartheta_1-\vartheta^*_1\|-\{\delta s\log(n/s)\}^{1/2}]^2) \\
\leq&\exp\left\{-\frac{[\|\vartheta_1-\vartheta_1^*\|-\{\delta s\log(n/s)\}^{1/2}]^2}{2+\delta_0}\right\}\leq \exp\left\{-\frac{\|\vartheta_1-\vartheta_1^*\|^2}{2+2\delta_0}\right\},
\end{split}
\end{equation}
where the last inequality holds as $\delta$ is sufficiently small.
Denote $\Delta\vartheta_1=\vartheta_1-\vartheta_1^*$.
With probability at least $1-\exp\{-cs\log(n/s)\}$ for some positive $c$, $\|\epsilon\|^2\leq {\eta s\log(n/s)/2}$ and
\begin{equation}\label{tfun3}
\begin{split}
&\frac{m(\tilde y)}{f^*(\tilde y)}=\int \exp\left(-\frac{-\|\epsilon\|^2+\|\Delta\vartheta_1+\epsilon\|^2}{2}\right)\pi(\vartheta_1)d\vartheta_1\\
\geq&\exp\left\{-\frac{\eta s\log(n/s)}{2}\right\}\pi(\|\Delta\vartheta_1\|\leq\{\eta s\log(n/s)/2\}^{1/2})
\end{split}
\end{equation}
for any fixed small constant $\eta>0$, where $m(\tilde y)=\int f(\tilde y;\vartheta_1)\pi(d\vartheta_1)$ is the marginal likelihood of data $\tilde y$.

Therefore, let $\Omega_n$ be the event that (\ref{tfun3}) holds, by (\ref{tfun1})-(\ref{tfun3}), we have
\[
\begin{split}
&E^*\pi(C_n|\tilde y)\\
=& E^*\pi(C_n|\tilde y)(1-\phi(\tilde y))1_{\Omega_n}+E^*\pi(C_n|\tilde y)(1-\phi(\tilde y))(1-1_{\Omega_n})+ E^*\pi(C_n|\tilde y)\phi(\tilde y)\\
\leq& E^*\pi(C_n|\tilde y)(1-\phi(\tilde y))1_{\Omega_n}+E^*(1-1_{\Omega_n})+ E^*\phi(\tilde y)\\
\leq& E^*\pi(C_n|\tilde y)(1-\phi(\tilde y))1_{\Omega_n}+\exp\{-cs\log(n/s)\}+\exp\{-\delta s\log (n/s)/(2+\delta_0)\}\\
\leq& \frac{E^*\int_{C_n}f(\tilde y;\vartheta_1)(1-\phi(\tilde y))/f^*(\tilde y)\pi(d\vartheta_1)}{\exp\left\{-\frac{\eta s\log(n/s)}{2}\right\}\pi(\|\Delta\vartheta_1\|\leq\{\eta s\log(n/s)/2\}^{1/2})}\\
&+\exp\{-cs\log(n/s)\} +\exp\{-\delta s\log (n/s)/(2+\delta_0)\}.
\end{split}
\]
And the first term in the above equation satisfies
\[
\begin{split}
&\frac{E^*\int_{C_n}f(\tilde y;\vartheta_1)(1-\phi(\tilde y))/f^*(\tilde y)\pi(d\vartheta_1)}{\exp\left\{-\frac{\eta s\log(n/s)}{2}\right\}\pi(\|\Delta\vartheta_1\|\leq\{\eta s\log(n/s)/2\}^{1/2})}\\
=&\frac{\int\int_{C_n}f(\tilde y;\vartheta_1)(1-\phi(\tilde y))\pi(d\vartheta_1)d\tilde y}{\exp\left\{-\frac{\eta s\log(n/s)}{2}\right\}\pi(\|\Delta\vartheta_1\|\leq\{\eta s\log(n/s)/2\}^{1/2})}\\
=&\frac{\int_{C_n}E_{\vartheta_1}(1-\phi(\tilde y))\pi(d\vartheta_1)}{\exp\left\{-\frac{\eta s\log(n/s)}{2}\right\}\pi(\|\Delta\vartheta_1\|\leq\{\eta s\log(n/s)/2\}^{1/2})}\\
\leq&\frac{\int_{C_n}\exp\{-\|\Delta\vartheta_1\|^2/(2+2\delta_0)\} \pi(d\vartheta_1)}{\pi(\|\Delta\vartheta_1\|\leq\{\eta s\log(n/s)/2\}^{1/2})}\exp\left\{\frac{\eta s\log(n/s)}{2}\right\}.
\end{split}
\]
Let us now study the quantity 
$\int_{C_n}\exp\{-\|\Delta\vartheta_1\|^2/(2+2\delta_0)\} \pi(d\vartheta_1)\big/\pi(\|\Delta\vartheta_1\|\leq\{\eta s\log(n/s)/2\}^{1/2})$.
Let $T_1=\{j;|\vartheta_{1,j}^*|\geq 1\}$ and $T_2=\{j;|\vartheta_{1,j}^*|< 1\}$, $T_3$ be the generic notation for a subset of $T_1$, $T_4=T_1\backslash T_3$
and $t_i=|T_i|$. 
$\vartheta_{1,T_i}$ denotes the subvector of $\vartheta_1$ corresponding to $T_i$. Decompose $C_n =\cup_{T_3\subset T_1}C_n^{T_3}:= \cup_{T_3\subset T_1}(C_n\cap\{\vartheta_1: \{j: |\vartheta_{1,j}|<1\}=T_3\})$.
Then, we have
\begin{equation}\label{bound0}
\begin{split}
&\frac{\int_{C_n}\exp\{-\|\Delta\vartheta_1\|^2/(2+2\delta_0)\} \pi(d\vartheta_1)}{\pi(\|\Delta\vartheta_1\|\leq\{\eta s\log(n/s)/2\}^{1/2})}\\
\leq&\frac{\int_{C_n}\exp\{-\|\Delta\vartheta_1\|^2/(2+2\delta_0)\} \pi(d\vartheta_1)}{\pi(\|\Delta\vartheta_{1,j}\|\leq\{\eta \log(n/s)/2\}^{1/2} \mbox{ for all }j=1,\dots, s)}\\
\leq &
\frac{\int_{C_n}\exp\{-\|\Delta\vartheta_1\|^2/(2+2\delta_0)\}\prod_{j\in T_1}\pi(\vartheta_{1,j})\prod_{j\in T_2}\pi(\vartheta_{1,j}) d\vartheta_1}{\prod_{j\in T_1}\pi(\vartheta_{1,j}^*)\prod_{j\in T_2}\pi(\vartheta_{1,j}\in[-1,1])}\\
=& \sum_{T_3\subset T_1}\frac{ \int_{C_n^{T_3}} \exp\{-\frac{\|\Delta\vartheta_{1}\|}{2+2\delta_0}\}\prod_{j\in T_4}\frac{\pi(\vartheta_{1,j})}{\pi(\vartheta_{1,j}^*)}\prod_{j\in T_3}\frac{1}{\pi(\vartheta_{1,j}^*)}\prod_{j\in T_2\cup T_3}\pi(\vartheta_{1,j}) d\vartheta_1}{\prod_{j\in T_2}\pi(\vartheta_{1,j}\in[-1,1]) },
%=& \sum_{T_3\subset T_1}\frac{ \int_{C_n^{T_3}} \exp\{-\frac{\|\Delta\vartheta_{1,T_2}\|^2+\|\Delta\vartheta_{1,T_3}\|^2+\|\Delta\vartheta_{1,T_4}\|^2}{2+2\delta_0}\}\prod_{j\in T_4}\frac{\pi(\vartheta_{1,j})}{\pi(\vartheta_{1,j}^*)}\prod_{j\in T_3}\pi(\vartheta_{1,j})\prod_{j\in T_2}\pi(\vartheta_{1,j}) d\vartheta_1}{\prod_{j\in T_3}\pi(\vartheta_{1,j}^*)\prod_{j\in T_2}\pi(\vartheta_{1,j}\in[-1,1]) },
\end{split}
\end{equation}
where the second inequality holds asymptotically, because $\log(n/s)$ is sufficiently large and $\pi(|\cdot|)$ is a decreasing function.

By C.3, if $|\vartheta_{1,j}|$ and $|\vartheta_{1,j}^*|$ are both larger than 1, then
$|\log(\pi(\vartheta_{1,j})/\pi(\vartheta_{1,j}^*))|\leq |\log(C_2/C_1)|+\alpha |\log|\vartheta_{1,j}|-\log|\vartheta_{1,j}^*||
\leq |\log(C_2/C_1)|+\alpha |\vartheta_{1,j}-\vartheta_{1,j}^*|$. 
If $|\vartheta_{1,j}^*|\geq1$, then we have that $\log[1/\pi(\vartheta_{1,j}^*)]\leq \alpha\log|\vartheta_{1,j}^*|+(\alpha-1)\log(1/\tau)-\log C_1
\leq \alpha|\vartheta_{1,j}^*|+(\alpha-1)\log(1/\tau)-\log C_1$.
Using these facts,  for any $T_3\subset T_1$,
\begin{equation}\label{boundrev1}
\begin{split}
 &\frac{\int_{C_n^{T_3}} \exp\{-\frac{\|\Delta\vartheta_{1}\|}{2+2\delta_0}\}\prod_{j\in T_4}\frac{\pi(\vartheta_{1,j})}{\pi(\vartheta_{1,j}^*)}\prod_{j\in T_3}\frac{1}{\pi(\vartheta_{1,j}^*)}\prod_{j\in T_2\cup T_3}\pi(\vartheta_{1,j}) d\vartheta_1}{\prod_{j\in T_2}\pi(\vartheta_{1,j}\in[-1,1])}\\
\leq &\frac{ \int_{C_n^{T_3}} \exp\{-\frac{\|\Delta\vartheta_{1}\|^2}{2+2\delta_0}+\alpha\|\Delta\vartheta_{1,T_4}\|_1+t_4\log\frac{C_2}{C_1}+\alpha\|\vartheta_{1,T_3}^*\|_1\}\prod_{j\in T_3\cup T_2}\pi(\vartheta_{1,j}) d\vartheta_1}
{C_1^{t_3}[\tau^{\alpha-1}]^{t_3}[1-2C_2(1/\tau)^{-(\alpha-1)}]^{t_2}}.
\end{split}
\end{equation}
Since for any  $j\in T_1$,
$\Delta\vartheta_{1,j}^2-(2+2\delta_0)\alpha|\Delta\vartheta_{1,j}|=(|\Delta\vartheta_{1,j}|-(1+\delta_0)\alpha)^2-
(1+\delta_0)^2\alpha^2$, we have that 
\begin{equation}\label{boundrev2}
\begin{split}
 &\frac{ \exp\{-\frac{\|\Delta\vartheta_{1}\|^2}{2+2\delta_0}+\alpha\|\Delta\vartheta_{1,T_4}\|_1+t_4\log\frac{C_2}{C_1}+\alpha\|\vartheta_{1,T_3}^*\|_1\}}
{C_1^{t_3}[\tau^{\alpha-1}]^{t_3}[1-2C_2(1/\tau)^{-(\alpha-1)}]^{t_2}}\\
\leq &\frac{  \exp\{-\frac{[\|\Delta\vartheta_1\|-s^{1/2}(1+\delta_0)\alpha]^2-s(1+\delta_0)^2\alpha^2}{2+2\delta_0}+t_4\log\frac{C_2}{C_1}+\alpha t_3\}}
{C_1^{t_3}[\tau^{\alpha-1}]^{t_3}[1-2C_2(1/\tau)^{-(\alpha-1)}]^{t_2} }\\
\leq &\frac{ \exp\{-\frac{[\{(2+\omega)s\log(n/s)\}^{1/2}-s^{1/2}(1+\delta_0)\alpha]^2-s(1+\delta_0)^2\alpha^2}{2+2\delta_0}+t_4\log\frac{C_2}{C_1}+\alpha t_3\}}
{C_1^{t_3}[\tau^{\alpha-1}]^{t_3}[1-2C_2(1/\tau)^{-(\alpha-1)}]^{t_2} }\\
\leq &\frac{  \exp\{-\kappa s\log(n/s)\}}
{[\tau^{\alpha-1}]^{t_3} }
\end{split}
\end{equation}
for any $0<\kappa< (2+\omega)/(2+2\delta_0)$, where the last inequality holds since $t_4$, $t_3<s$, $\tau\rightarrow 0$ and $n/s\rightarrow\infty$. Combining inequalities (\ref{bound0})-(\ref{boundrev2}), (\ref{bound0}) can be bounded by 
\begin{equation}\label{bound1}
\begin{split}
&\frac{\int_{C_n}\exp\{-\|\Delta\vartheta_1\|^2/(2+2\delta_0)\} \pi(d\vartheta_1)}{\pi(\|\Delta\vartheta_1\|\leq\{\eta s\log(n/s)/2\}^{1/2})}\\
\leq &\sum_{T_3\subset T_1}\frac{ \int_{C_n^{T_3}} \exp\{-\kappa s\log(n/s)\}\prod_{j\in T_2\cup T_3}\pi(\vartheta_{1,j}) d\vartheta_1}
{[\tau^{\alpha-1}]^{t_3} }\\
\leq &\sum_{T_3\subset T_1}\int_{C_n^{T_3}}\left\{\frac{\exp[-\kappa \log(n/s)]}{[\tau^{\alpha-1}]}\right\}^{t_3}\prod_{j\in T_3}\pi(\vartheta_{1,j})\prod_{j\in T_2}\pi(\vartheta_{1,j}) d\vartheta_1\\
\leq &\int_{\|\Delta\vartheta_{1,T_1}\|\leq\{(2+\omega)s\log(n/s)\}^{1/2}}\prod_{j\in T_1}\tilde\pi(\vartheta_{1,j}) d\vartheta_{1,T_1}\\
\leq &\frac{\exp\{-\kappa s\log(n/s)\}}{[\tau^{\alpha-1}]^{s}}[1+V_s(\{(2+\omega)s\log(n/s)\}^{1/2})],
\end{split}
\end{equation}
where $\tilde \pi(\vartheta)=\exp[-\kappa\log(n/s)]\pi(\vartheta)/\tau^{\alpha-1}$ if $|\vartheta|< 1$, $\tilde \pi(\vartheta)=1$ if $|\vartheta|\geq 1$;
and $V_n(R)$ is the volume of $n$-dimensional ball with radius $R$.

Combining all the above calculus results, if $\tau^{\alpha-1}\geq (s/n)^{c}\{\log(n/s)\}^{1/2}$ for $0<c<\kappa-\eta/2$ (which is guaranteed by the condition of the theorem, as long as $\delta_0$ and $\eta$ are sufficiently small), then
\[
\frac{\int_{C_n}\exp\{-\|\Delta\vartheta_1\|^2/(2+2\delta_0)\} \pi(d\vartheta_1)}{\pi(\|\Delta\vartheta_1\|\leq\{\eta s\log(n/s)/2\}^{1/2})}\exp\{\frac{\eta}{2} s\log(n/s)\}\leq
\exp\{-c's\log(n/s)\}
\]
for some $0<c<\kappa-\eta/2-c$. And this concludes (\ref{part1}).

\vskip 0.1in
{\noindent Part II: posterior contraction rate of zero $\theta_i$'s}
\vskip 0.1in

It is equivalent to consider the situation that 
$\tilde y=\vartheta_2+\epsilon$ where $\tilde y,\vartheta_2\in \BR^{n-s}$, $\epsilon\sim N(0,I_{n-s})$.
The parameter $\vartheta_2$ is subject to prior $\prod_{j=1}^{n-s}\pi(\vartheta_{2,j})$. 
The true parameter $\vartheta_2^*=0$, and we want to show that 
\begin{equation}\begin{split}\label{part2}
&E^*\pi(\|\vartheta_2\|\geq \sqrt{\omega s\log(n/s)}, \mbox{at most }c\delta s \mbox{ entries of }\vartheta_2 \\
& \qquad\mbox{ is larger than } \sqrt{\delta s\log(n/s)/n}|\tilde y)
\leq \exp\{-c_0s\log(n/s)\},
\end{split}
\end{equation}
if $\tau^{\alpha-1}\prec [(s/n)\log(n/s)]^{\alpha}$; 
and
\begin{equation}\begin{split}\label{part3}
&E^*\pi(\|\vartheta_2\|\geq \sqrt{\omega s\log(n/s)}^, \mbox{at most }c\delta s \mbox{ entries of }\vartheta_2 \\
& \qquad \mbox{ is larger than } s\sqrt{\delta \log(n/s)}/n|\tilde y)\leq \exp\{-c_0s\log(n/s)\},
\end{split}
\end{equation}
if $\tau^{\alpha-1}\prec(s/n)^\alpha\{\log(n/s)\}^{(\alpha+1)/2}$,
for %some tiny constant $\omega$, 
$\delta=\omega/5$ and some constants $c<1/2$ and $c_0>0$. We will apply Lemma \ref{lemmab} to prove (\ref{part2}) and (\ref{part3}).

To proof (\ref{part2}), we consider the testing function $\phi(\tilde y)=\max_{|\xi|\leq c\delta s}1(\|\tilde y_\xi\|\geq \{\delta s\log(n/s)\}^{1/2})$,
where  $\tilde y_{\xi}$ is the subvector of $\tilde y$ corresponding to model $\xi$.
First, for any fixed $\delta_0>0$, by Lemma \ref{chi}
and Sterling's approximation, we have
\begin{equation}\label{tfun4}
\begin{split}
E_{\vartheta_2^*}\phi(\tilde y)&= \lceil c\delta s\rceil{n \choose {\lceil c\delta s\rceil} }Pr(\chi_{\lceil c\delta s\rceil}^2\geq \delta s\log(n/s))\\
&\leq \lceil c\delta s\rceil{n \choose {\lceil c\delta s\rceil} }\exp\left\{-\frac{\delta s\log (n/s)}{2+\delta_0}\right\}\leq \exp\{-c's\log(n/s)\}
\end{split}
\end{equation}
for some $0<c'<\delta(1/(2+\delta_0)-c)$, when $n$ is sufficiently large and we choose $c<1/(2+\delta_0)$.

We define two sets in $\BR^{n-s}$ as: $B_n=\{\mbox{more than }c\delta s\mbox{ entries of }|\vartheta_2| \mbox{ are bigger than }$ $\sqrt{\delta s\log(n/s)/n}\}$,
and 
$C_n=\{\|\vartheta_2\|\geq \sqrt{5\delta s\log(n/s)} \mbox{ and at most }c\delta s\mbox{ entries of }|\vartheta_2|$ $\mbox{ are bigger than }\sqrt{\delta s\log(n/s)/n}\}$.
For any $\vartheta_2\in C_n$, let $\hat\xi=\hat\xi(\vartheta_2)=\{j: |\vartheta_{2,j}|\geq\{\delta s\log(n/s)/n\}^{1/2}\}\}$, thus we always have that
$|\hat\xi|\leq c\delta s$, $\|\vartheta_{2,\hat\xi^c}\|\leq \{\delta s\log(n/s)\}^{1/2}$
and $\|\vartheta_{2,\hat\xi}\|\geq 2\{\delta s\log(n/s)\}^{1/2}$. 
Then we can derive that
\begin{equation}\label{tfun5}
\begin{split}
&\sup_{\vartheta_2\in C_n}E_{\vartheta_2}[1-\phi(\tilde y)]\leq Pr(\|\tilde y_{\hat\xi}\|\leq \{\delta s\log(n/s)\}^{1/2}|\vartheta_2)\\
\leq& Pr(\|\tilde y_{\hat\xi}-\vartheta_{2,\hat\xi}\|\geq \|\vartheta_{2,\hat\xi}-\vartheta^*_{2,\hat\xi}\|-\{\delta s\log(n/s)\}^{1/2}|\vartheta_2) \\
\leq & Pr(\chi_{\lceil c\delta s\rceil}^2\geq[\|\vartheta_{2,\hat\xi}-\vartheta^*_{2,\hat\xi}\|-\{\delta s\log(n/s)\}^{1/2}]^2) \\
\leq&\exp\left\{-\frac{\delta s\log(n/s)}{2+\delta_0}\right\}.
\end{split}
\end{equation}

Note that with dominating probability, $\|y\|\leq (1+c') n^{1/2}$ for any $c'>0$ and 
\begin{equation}\label{margin}
\begin{split}
&\frac{m(\tilde y)}{f^*(\tilde y)}\geq \int_{\|\vartheta_2\|_{\infty}\leq s\log (n/s)/n}\frac{f(\vartheta_2;\tilde y)}{f(0;\tilde y)}\pi(\vartheta_2)\\
= &\int_{\|\vartheta_2\|_{\infty}\leq \eta s\log (n/s)/n}\exp\left\{-\frac{1}{2}(\|\tilde y-\vartheta_2\|^2-\|\tilde y\|^2)\right\}\pi(\vartheta_2)\\
\geq &\exp\{-3\eta s\log(n/s)\}\pi(\|\vartheta_2\|_{\infty}\leq \eta s\log (n/s)/n),
\end{split}
\end{equation}
for any positive $\eta$.
By the condition of $\tau$, 
\begin{equation}\label{margin2}
\begin{split}
&\pi(\|\vartheta_2\|_{\infty}\leq \eta s\log (n/s)/n) \geq(1-C_2(\frac{\eta \frac{s}{n}\log\frac{n}{s}}{\tau})^{-(\alpha-1)})^{n-s}\\
\geq &(1-\eta'\frac{s}{n}\log\frac{n}{s})^{n-s}\rightarrow \exp\{-\eta's\log(n/s)\}
\end{split}
\end{equation}
for any positive $\eta'$.

Similar, we have the $\pi(|\vartheta_{2,i}|\geq\{\delta s\log(n/s)/n\}^{1/2})=o([(s/n)\log(n/s)]^{(\alpha+1)/2})$. Thus by Lemma \ref{lemmad}, 
it is easy to verify that the prior of $B_n$ satisfies
\begin{equation}\label{bn}
\begin{split}
-\log\pi(B_n)\gtrsim c\delta [(\alpha-1)/2] s\log(n/s)
\end{split}
\end{equation}

Combining results (\ref{tfun4})-(\ref{bn}), by Lemma \ref{lemmab}, one can see that
(\ref{part2}) holds as long as we choose sufficiently small $\eta$ and $\eta'$.

Similar arguments can be use to proof (\ref{part3}). The only difference is that we now define the set
$B_n=\{\mbox{more than }c\delta s\mbox{ entries of }|\vartheta_2| \mbox{ are bigger than }s\sqrt{\delta \log(n/s)}/n\}$ and set
$C_n=\{\|\vartheta_2\|\geq \sqrt{5\delta s\log(n/s)} \mbox{ and at most }c\delta s\mbox{ entries of }|\vartheta_2| \mbox{ are bigger than }$ $s\sqrt{\delta \log(n/s)}/n\}$.
The details of proving (\ref{part3}) is left to the readers.

\vskip 0.2in
Due to the fact that posterior of $\vartheta_1$ and $\vartheta_2$ are independent, (\ref{part1})
and (\ref{part2}) imply 
\[\begin{split}
&E^*\pi(\|\theta-\theta^*\|\geq \{(2+\omega)s\log(n/s)\}^{1/2}+\{(\omega)s\log(n/s)\}^{1/2}|D_n)\\
\leq& \exp\{-c_0's\log(n/s)\}
\end{split}\] for some $c_0'$.
And (\ref{part1}) and (\ref{part3}) imply that 
\[
\begin{split}
E^*\pi(\|\theta-\theta^*\|_1\geq s\{(2+\omega)&\log(n/s)\}^{1/2}+s\{\omega\delta\log(n/s)\}^{1/2} \\
&+s\{\delta\log(n/s)\}^{1/2}|D_n)
\leq \exp\{-c_0's\log(n/s)\},
\end{split}
\]for some $c_0'$.
These conclude our results in Theorem \ref{main}.

\subsection{Proof of Theorem \ref{adaptive}}
Consider the testing function
\begin{equation}\label{test}
\phi(D_n)=\max_{\xi\supset\xi^*,|\xi\backslash\xi^*|\leq c\delta s}1(\|y_\xi-\theta_\xi^*\|\geq \{\delta s\log(n/s)\}^{1/2}),
\end{equation}
for some $c\leq 1/2$,
and define two sets in $\Theta$: $B_n=\{\theta: |\{j: j\notin\xi^*, |\theta_j|\geq s\{\delta\log(n/s)\}^{1/2}/n\}|\geq c\delta s\}$, and $C_n=\{\theta:\|\theta-\theta^*\|\geq\sqrt{(2+2\omega)s\log(n/s)}\}\backslash B_n$,
where $\xi^*=\{j: |\theta^*_j|\neq 0\}$.
The $\delta$ is a small quantity depending on $\omega$ which we will determine later.

By the same arguments used in the proof of Theorem \ref{main} and Lemma \ref{chi}, 
we have that, for any fixed small $\delta_0$ satisfying $1/(2+\delta_0)>c$,
\[E_{\theta^*}\phi(D_n)\leq \exp\left\{-[\frac{1}{2+\delta_0}-c]\delta s\log (n/s)\right\},\]
if $n$ is sufficiently large,
and 
\[\sup_{\theta\in C_n}E_{\theta}(1-\phi(D_n))\leq\exp\{-[\{ (2+2\omega-\delta)^{1/2}-\delta^{1/2}\}^2/(2+\delta_0)]s\log(n/s)\},\]
and we choose the values of $\delta$ and $\delta_0$ to be very small,  such that 
$\{ \sqrt{2+2\omega-\delta}-\sqrt\delta\}^2/(2+\delta_0)>1+3\omega/4$.

To derive the upper bound for $E^*\pi(C_n\cup B_n|D_n)$, let us study $E^*\pi(C_n|D_n)$ and $E^*\pi(B_n|D_n)$ separately.

Use the same notation and arguments in the proof of Theorem \ref{main}, let $m(D_n)$ be the marginal density 
of data, and $f^*(D_n)$ be the true likelihood.
Let $\vartheta_1$ and $\vartheta_2$ be the subvectors of $\theta$ corresponding to $\xi^*$ and 
$\xi^{*c}$.
Thus with probability $1-\exp\{-cs\log(n/s)\}$,
\[\begin{split}
\frac{m(D_n)}{f^*(D_n)}\geq &\exp\{-4\eta s\log(n/s)\}\\
&\times\pi(\|\vartheta_1-\vartheta_1^*\|\leq \{\eta s\log(n/s)/2\}^{1/2}, \|\vartheta_2\|_\infty\leq \eta s\log(n/s)/n)
\end{split}
\]
for any positive $\eta$.

By Lemma \ref{lemmab} and testing function (\ref{test}),
we have that 
\begin{equation}\label{tfun6}
\begin{split}
&E^*\pi(C_n|D_n)\leq \exp\{-cs\log(n/s)\}+\exp\left\{-[\frac{1}{2+\delta_0}-c]\delta s\log (n/s)\right\}\\
+&\frac{\exp\{-[\{ (2+2\omega-\delta)^{1/2}-\delta^{1/2}\}^2/(2+\delta_0)]s\log(n/s)\}}{\exp\{-4\eta s\log(n/s)\}\pi(\|\vartheta_1-\vartheta_1^*\|\leq \sqrt{\eta s\log(n/s)/2}, \|\vartheta_2\|_\infty\leq \eta s\log(n/s)/n)}.
\end{split}
\end{equation}

When $\tau\in[(s/n)^{(1+\omega/2)/(\alpha-1)},(s/n)^{\alpha/(\alpha-1)}]$,
as showed in the proof of Theorem \ref{main}, we have that $\pi(\|\vartheta_2\|_\infty\leq \eta s\log(n/s)/n)|\tau)
\geq \exp\{-\eta's\log(n/s)\}$ for any positive $\eta'$,
and 
\[\begin{split}
&\pi(\|\vartheta_1-\vartheta_1^*\|\leq \{\eta s\log(n/s)/2\}^{1/2}|\tau)\\
\geq& \{\eta\log(n/s)/2\}^{s/2} \left(\min_{|\theta|\leq\{\eta\log(n/s)/2\}^{1/2}+\max|\theta_j^*| } \pi(\theta|\tau)\right)^s\\
\geq &\exp\{-(1+\omega/2+\omega/5+\eta'')s\log(n/s)\},
\end{split}
\]
for any positive $\eta''$, where the last inequality is due to the upper bound condition
of $\max|\theta_j^*|$.
Therefore, 
\[
\begin{split}
&\pi\{\|\vartheta_1-\vartheta_1^*\|\leq \sqrt{\eta s\log(n/s)/2}, \|\vartheta_2\|_\infty\leq \eta s\log(n/s)/n)\}\\
\geq& \pi\{(s/n)^{(1+\omega/2)/(\alpha-1)}\leq \tau\leq (s/n)^{\alpha/(\alpha-1)}\}
%\exp\{-\eta's\log(n/s)\}
\\
&\times\exp\{-(\eta'+1+\omega/2+\omega/5+\eta'')s\log(n/s)\}.
\end{split}
\]
Combining the above results, with the condition on the prior $\pi(\tau)$, we have that
(\ref{tfun6})$\leq \exp\{-c's\log(n/s)\}$ for some positive $c'$, given $\eta'$ and $\eta''$
are sufficiently small.

Now we study the posterior $E^*\pi(B_n|D_n)$.
The marginal distribution can be written as $m(D_n)=\int f(y^{(1)};\vartheta_1)f(y^{(2)};\vartheta_2)\pi(\theta)d\theta$, where $y^{(1)}=y_{\xi^*}$, $y^{(2)}=y_{\xi^{*c}}$, 
$ f(y^{(1)};\vartheta_1)$ and $f(y^{(2)};\vartheta_2)$ are the likelihood functions for $y^{(1)}$ and $y^{(2)}$, and $\theta=(\vartheta_1,\vartheta_2)$.
By the same arguments used in the proof of Theorem \ref{main},
with probability $1-\exp\{-cs\log(n/s)\}$,
\begin{equation}\begin{split}\label{set}
&m(D_n)\geq f(y^{(2)};\vartheta^*_2)\exp\{-3\eta s\log(n/s)\}\int_{\|\vartheta_2\|_\infty\leq \eta s\log(n/s)/n} f(y^{(1)};\vartheta_1)\pi(\theta)d\theta,\\
&\mbox{and }\|y^{(1)}-\vartheta_1^*\|^2\leq \eta''' s\log(n/s),
\end{split}
\end{equation}
for some positive $c$ and $\eta'''$. Let $\Omega_n$ be the event that (\ref{set}) holds.
%Let $\Omega_n$ be the event $\{\|y^{(1)}\|_\infty\leq 2(n/s)^{\omega/5}\}$.

Therefore, by the same argument in the Lemma \ref{lemmab},
\begin{equation}\label{tfun7}
\begin{split}
&E^*\pi(B_n|D_n)\\
\leq& %\exp\{-cs\log(n/s)\}+
(1-P^*(\Omega_n))\\
& + E_{y^{(1)}}\frac{ E_{y^{(2)}}\int_{B_n} f(y^{(1)};\vartheta_1)f(y^{(2)};\vartheta_2)\pi(\theta)d\theta /f(y^{(2)};\vartheta_2^*)}
{ \exp\{-3\eta s\log(n/s)\}\int_{\|\vartheta_2\|_\infty\leq \eta s\log(n/s)/n} f(y^{(1)};\vartheta_1)\pi(\theta)d\theta}1_{\Omega_n}\\
=& %\exp\{-cs\log(n/s)\} + 
(1-P^*(\Omega_n))\\ 
&+E_{y^{(1)}}\frac{ \int_{B_n} f(y^{(1)};\vartheta_1)\pi(\theta)d\theta }
{ \exp\{-3\eta s\log(n/s)\}\int_{\|\vartheta_2\|_\infty\leq \eta s\log(n/s)/n} f(y^{(1)};\vartheta_1)\pi(\theta)d\theta}1_{\Omega_n}.
\end{split}
\end{equation}

Let's  study the last term in the right handed side of (\ref{tfun7}).
Define two sets in $\BR^{n-s}$: $B_n^1=\{\vartheta_2: |\{j: |\vartheta_{2,j}|\geq s\{\delta\log(n/s)\}^{1/2}/n\}|\geq c\delta s\}$,
$B_n^2=\{\vartheta_2: \|\vartheta_2\|_\infty\leq \eta s\log(n/s)/n\}$, hence
\begin{equation}
\begin{split}
& E_{y^{(1)}}\frac{ \int_{B_n} f(y^{(1)};\vartheta_1)\pi(\theta)d\theta}{\int_{\|\vartheta_2\|_\infty\leq \eta s\log(n/s)/n} f(y^{(1)};\vartheta_1)\pi(\theta)d\theta}1_{\Omega_n}\\
\leq &  E_{y^{(1)}}\frac{(\int_{\tau\leq\tau_0}+\int_{\tau>\tau_0})\int_{B_n^1}\int_{\BR^s}f(y^{(1)};\vartheta_1)\pi(\vartheta_1|\tau)\pi(\vartheta_2|\tau)\pi(\tau)d\vartheta_1d\vartheta_2 d\tau  }
{\int_{\tau\leq\tau_0}\int_{B_n^2}\int_{\BR^s}f(y^{(1)};\vartheta_1)\pi(\vartheta_1|\tau)\pi(\vartheta_2|\tau)\pi(\tau)d\vartheta_1d\vartheta_2 d\tau }1_{\Omega_n},
\end{split}
\end{equation}
where $\tau_0=(s/n)^{\alpha/(\alpha-1)}$.
When $\tau\leq \tau_0$, by the same arguments used in the Part II of the proof of theorem \ref{main}, we have
\begin{equation}\begin{split}
&\frac{\int_{B_n^1}\int_{\BR^s}f(y^{(1)};\vartheta_1)\pi(\vartheta_1|\tau)\pi(\vartheta_2|\tau)d\vartheta_1d\vartheta_2  }
{\int_{B_n^2}\int_{\BR^s}f(y^{(1)};\vartheta_1)\pi(\vartheta_1|\tau)\pi(\vartheta_2|\tau)d\vartheta_1d\vartheta_2 }=\frac{\int_{B_n^1}\pi(\vartheta_2|\tau)d\vartheta_2 }
{\int_{B_n^2}\pi(\vartheta_2|\tau)d\vartheta_2 }\\
\leq&\frac{\exp\{-c\delta(\alpha-1)s\log(n/s)/2\}}
{\exp\{-\eta's\log(n/s)\}}
\end{split}\end{equation}
for any fixed small $\eta'>0$.
And on event $\Omega_n$,
\[\begin{split}
&\int_{\BR^s}\exp\{-\|y^{(1)}-\vartheta_1\|^2/2\}\pi(\vartheta_1|\tau)d\vartheta_1\\
=&\int_{\BR^s}\exp\{-\|y^{(1)}-\vartheta_1^*+\vartheta_1^*-\vartheta_1\|^2/2\}\pi(\vartheta_1|\tau)d\vartheta_1\\
\geq&\int_{\|\vartheta_1-\vartheta_1^*\|_\infty\leq \{\log(n/s)\}^{1/2}}\exp\{-\|y^{(1)}-\vartheta_1\|^2/2\}\pi(\vartheta_1|\tau)d\vartheta_1\\
\geq&\int_{\|\vartheta_1-\vartheta_1^*\|_\infty\leq \{\log(n/s)\}^{1/2}}\exp[-\{\eta'''+1+2(\eta''')^{1/2}\}s\log(n/s)/2]\pi(\vartheta_1|\tau)d\vartheta_1\\
=&\exp\{-c's\log(n/s)\}\int_{\|\vartheta_1-\vartheta_1^*\|_\infty\leq \{\log(n/s)\}^{1/2}}\pi(\vartheta_1|\tau)d\vartheta_1,
\end{split}\]
for some positive $c'$.
Therefore, conditional on $\Omega_n$, we have
\begin{equation}\label{tfun8}
\begin{split}
&\frac{\int_{\tau>\tau_0}\int_{B_n^1}\int_{\BR^s}f(y^{(1)};\vartheta_1)\pi(\vartheta_1|\tau)\pi(\vartheta_2|\tau)\pi(\tau)d\vartheta_1d\vartheta_2 d\tau  }
{\int_{\tau\leq\tau_0}\int_{B_n^2}\int_{\BR^s}f(y^{(1)};\vartheta_1)\pi(\vartheta_1|\tau)\pi(\vartheta_2|\tau)\pi(\tau)d\vartheta_1d\vartheta_2 d\tau }\\
\leq &\frac{\int_{\tau>\tau_0}\pi(\tau) d\tau  }
{\exp\{-c's\log(n/s)\}\int_{\tau_1}^{\tau_0}\int_{B_n^2}\int_{\|\vartheta_1-\vartheta_1^*\|_\infty\leq \sqrt{\log(n/s)}}\pi(\vartheta_1|\tau)\pi(\vartheta_2|\tau)\pi(\tau)d\vartheta_1d\vartheta_2 d\tau }\\
\leq &\frac{\int_{\tau>\tau_0}\pi(\tau) d\tau  }
{\exp\{-c''s\log(n/s)\}\int_{\tau_1\leq\tau\leq\tau_0}\pi(\tau)d\tau }\\
\leq &\exp\{-c'''s\log(n/s)\}
\end{split}
\end{equation}
for any positive $c''$ and $c'''$, where $\tau_1=(s/n)^{(1+\omega/2)/(\alpha-1)}$,
and the second inequality follows by the condition $\max|\vartheta_j^*|=\max|\theta_j^*|\leq (n/s)^{\omega/(5\alpha)}$.

Combining (\ref{tfun7})-(\ref{tfun8}), we have that $E^*\pi(B_n|D_n)\leq \exp\{-c''''s\log(n/s)\}$ for some positive $c''''$ if $\eta$ and $\eta'$ are small enough.

These results conclude the theorem.

\vskip 0.2in
\bibliographystyle{imsart-nameyear}
\bibliography{ref}

\end{document}

\thispagestyle{empty}
\title{Supplementary Material for ``Bayesian Shrinkage towards Sharp Minimaxity''}
\author{Qifan Song}
\date{}

\maketitle
\section{Additional simulations}
Three more simulation experiments are presented in this section as a supplementary to the numerical results in the main manuscript.

In the first experiment, we consider that the true nonzero coefficients are varying, rather than being constants. To be more specific, we let the data dimension increases as $n=50$, 100, 500, 1000 and the sparsity $s$ equals to the rounded value of ${n}^{1/2}$.
The nonzero coefficients are chosen to be $\theta^*_i=\{t\log(n/s)\}^{1/2}$ for $1\leq i\leq s$, where $t\sim\mbox{Unif}(0,5)$. We implement the $t$-prior:
\begin{equation}
\theta_i\sim \mbox{N}(0,\lambda_i^2\tau^2);\,
\lambda_i^2\sim \mbox{IG}((\alpha-1)/2, (\alpha-1)/2),
\end{equation}
with $\tau$ follows a deterministic choice $\tau=(s/n)^c$, as well as the horseshoe prior.
Four different choices of prior modelings are used for comparison:
(1) $\alpha=1.1$, $c=(\alpha+0.05)/(\alpha-1)$;
(2) $\alpha=2.1$, $c=(\alpha+0.05)/(\alpha-1)$;
(3) $\alpha=2.1$, $c=1/(\alpha-1)$;
(4) horseshoe prior with global shrinkage $\tau=(s/n)\{\log(n/s)\}^{1/2}$.
Discussion on the choices of $c$ can be found in the simulation section in the main manuscript. We evaluate the posterior mean of squared $L_2$ error (i.e., $E(\|\theta-\theta^*\|^2|D_n)$), posterior mean of $L_1$ error (i.e., $E(\|\theta-\theta^*\|_1|D_n)$) and the posterior probability outside $L_2$ minimax ball (i.e., $\pi(\|\theta-\theta^*\|^2\geq 2.2s\log(n/s)|D_n)$). The simulation results, based on 100 repetitions, are displayed in Figure \ref{mixed}.
The results clearly show that the $t$-prior with $\alpha=1.1$, $c=(\alpha+0.05)/(\alpha-1)$ performs the best, its posterior $L_2$ and $L_1$ errors are below the $L_2/L_1$ minimax rates respectively, and its posterior probability outside $\{\|\theta-\theta^*\|^2\geq 2.2s\log(n/s)\}$ decreases to zero fastest, as $n$  increases. Consistent to the simulation results presented in the main manuscript, horseshoe prior and $t$-prior with $\alpha=2.1$, $c=1/(\alpha-1)$ (green and blue curves in Figure \ref{mixed}) share very similar performance.

\begin{figure}[htbp]
	\begin{center}
		\includegraphics[width=7.5cm]{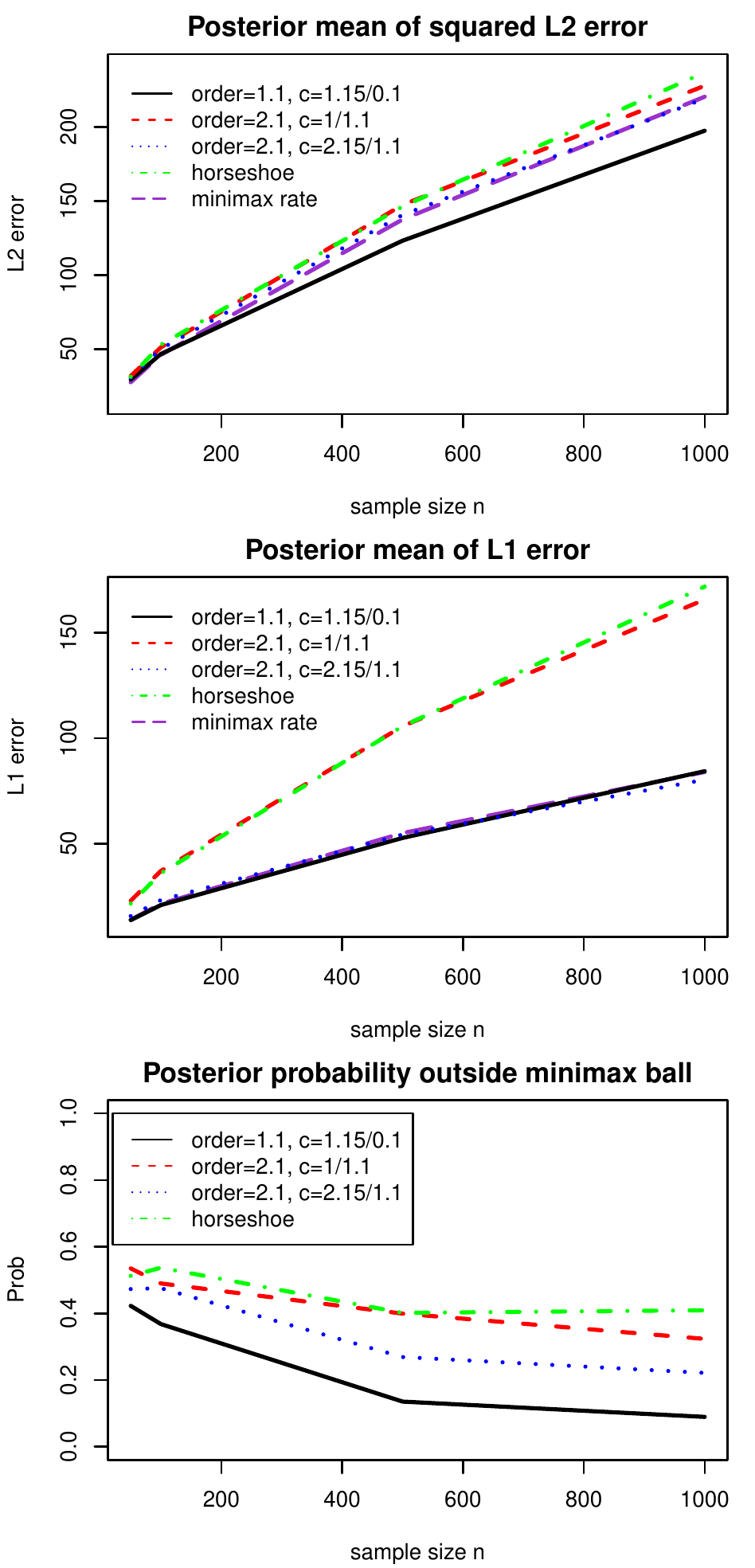}
		\caption{
			Posterior asymptotics under varying coefficients.
		}\label{mixed}
	\end{center}
\end{figure}

Our second experiment try to investigate the posterior convergence on the ``active'' $\theta_i$'s and ``nonactive'' $\theta_i$'s, i.e., 
$\theta_S$ and $\theta_{S^c}$, where $S=\{i:\theta_i^*\neq 0\}$ and the subscript means the corresponding subvector. Especially, we would like to compare $t$-prior with $\alpha=1.1$, $c=(\alpha+0.05)/(\alpha-1)$ against horseshoe prior
with global shrinkage $\tau=(s/n)\{\log(n/s)\}^{1/2}$. As discussed, the lower bound condition on $\tau$ in Theorem 2.1 ensures posterior convergence of the nonzero $\theta$'s, and the upper bound condition on $\tau$ ensures the posterior convergence of the zero $\theta$'s. The choice of $\tau$ for horseshoe prior only satisfies the lower bound condition, but not the upper bound condition, hence we  conjecture that horseshoe prior leads to much worse posterior contraction on the nonactive $\theta$'s. Our simulation result indeed justifies it. 
The simulation dimension increases as $n=50$, 100, 500, 1000 and the sparsity $s$ equals to the rounded value of ${n}^{1/2}$. The nonzero coefficients are chosen to be $\theta^*_i=\{t\log(n/s)\}^{1/2}$ for $1\leq i\leq s$, where $t=1.2, 2.2, 4.2$.
In Figure \ref{rate}, we plot $E(\|\theta_S-\theta_S^*\|^2|D_n)$ (dashed line, left panel), $E(\|\theta_{S^c}-\theta_{S^c}^*\|^2|D_n)$ (dash-dot line, left panel), $E(\|\theta_S-\theta_S^*\|_1|D_n)$ (dashed line, right panel) and $E(\|\theta_{S^c}-\theta_{S^c}^*\|_1|D_n)$ (dash-dot line, right panel). It clearly shows that there is little difference between $t$-prior and horseshoe prior, in terms of the posterior contraction on the active $\theta$'s, regardless of the signal strength value $t$. However, horseshoe has a much worse posterior contraction rate on the nonactive $\theta$'s than $t$-prior, under both $L_1$ and $L_2$  metrics.

\begin{figure}[htbp]
	\begin{center}
		\includegraphics[width=13cm]{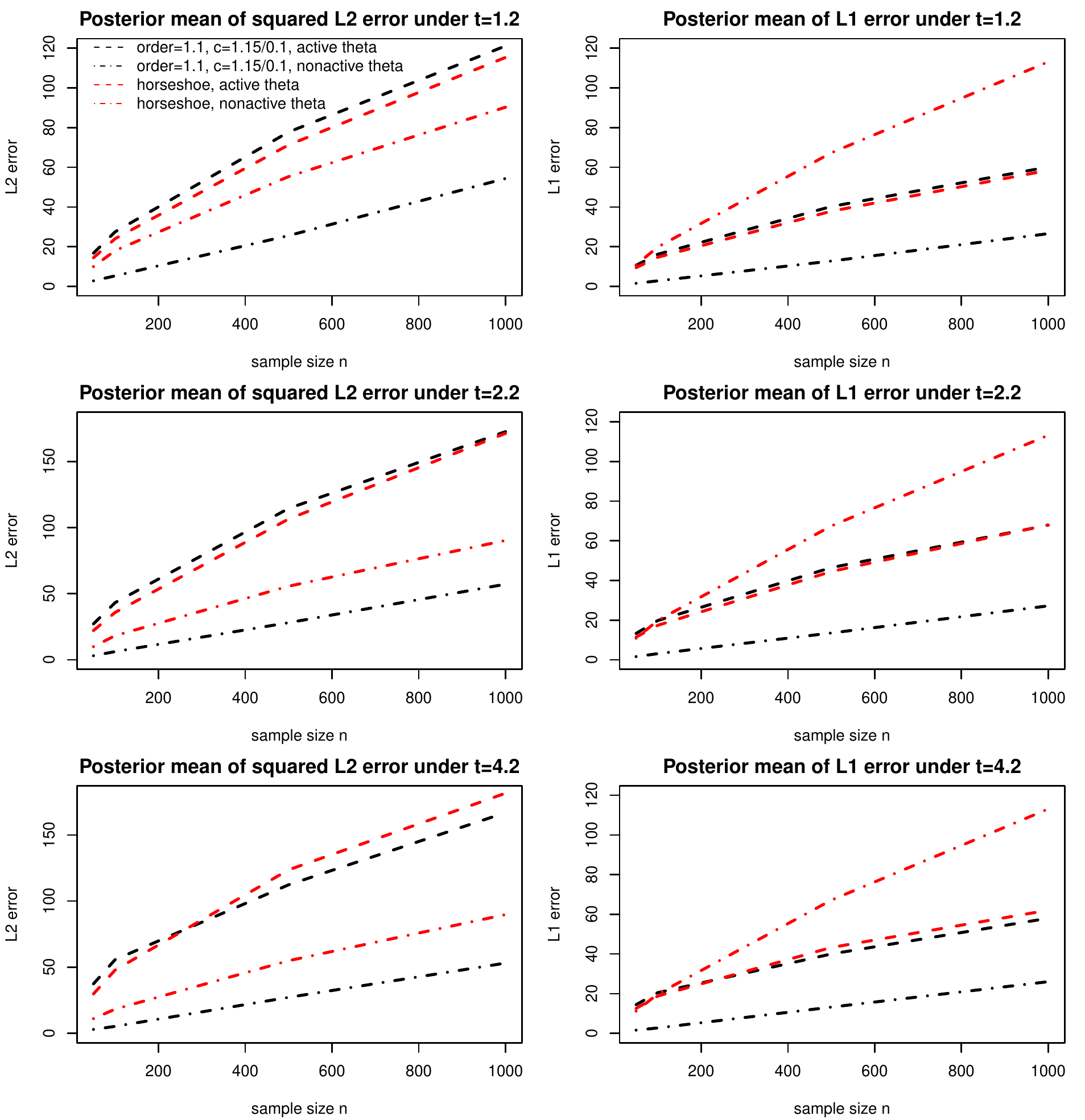}
		\caption{
			Posterior contraction for active and nonactive $\theta_i$'s. The results for $t$-prior are in black color, and the results for horseshoe are in red color.
		}\label{rate}
	\end{center}
\end{figure}

In the end, we perform an exploratory experiment to study the inference performance (i.e., credible intervals and model selection) based on shrinkage priors. In the literature, \citet{VanSV2017_2} investigated uncertainty quantification of horseshoe prior. In this experiment, we consider the marginal credible interval for $\theta_i$ of form $E(\theta_i|D_n)\pm 1.96\sqrt{Var(\theta_i|D_n)}$ (the same interval was used in the simulation studies of \citep{VanSV2017_2}). And the Bayesian model selection result is based on whether the marginal credible interval includes zero or not. We assess the coverage of marginal credible intervals and the correctness of model selection results, for $t$-prior with $\alpha=1.1$, $c=(\alpha+0.05)/(\alpha-1)$ and horseshoe prior with global shrinkage $\tau=(s/n)\{\log(n/s)\}^{1/2}$. Namely, we evaluate the percentage of selected true nonzero $\theta_i$'s and the percentage of non-selected zero $\theta_i$'s (for both percentages, the higher the better), as well as the average credible interval coverages for both nonzero and zero $\theta_i$'s. The simulation dimension increases as $n=50$, 100, 500, 1000 and the sparsity $s$ equals to the rounded value of ${n}^{1/2}$. The nonzero coefficients are chosen to be $\theta^*_i=\{t\log(n/s)\}^{1/2}$ where $t=1.2, 2.2, 4.2$. The results, based on 100 repetitions, are displayed in Figure \ref{inference}.

\begin{figure}[htbp]
	\begin{center}
		\includegraphics[width=13cm]{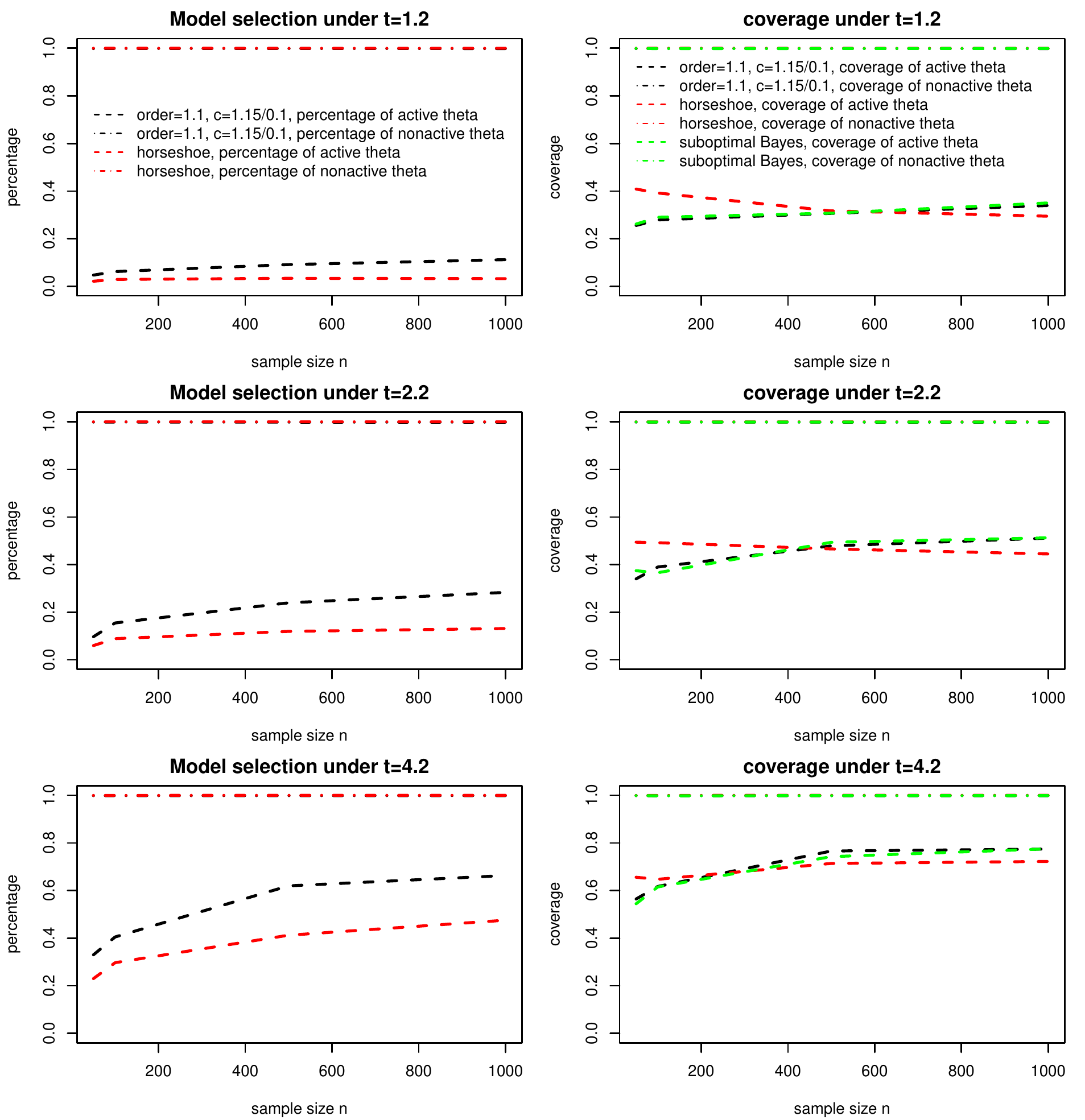}
		\caption{
			Posterior inferences for shrinkage priors. Left panels: model selection accuracy. Right panels: credible interval coverages. The results of $t$-prior are in black color and the results of horseshoe prior are in red color.
		}\label{inference}
	\end{center}
\end{figure}

For both model selection and marginal credible interval coverage, the zero $\theta_i$'s are almost never selected, and have almost 100\% coverage. On the other hand, for the nonzero $\theta_i$'s, their probability to be selected and  marginal credible interval coverages, heavily relies on the signal strength. That is, larger strength leads to better performance. This observation is consistent to the theoretical results discovered by \citep{VanSV2017_2}: consistent model selection and interval coverage don't happen to medium signal $\theta_i$'s.
Comparing with horseshoe prior, we observe that $t$-prior performs better for  model selection, but has a comparable performance of credible interval coverage.
In additional, we also study the credible interval coverage by $t$-prior with $\alpha=1.1$, $\tau=(1/n)^{c}$ where $c=(\alpha+0.05)/(\alpha-1)$. Due to Theorem 2.2, we believe it will lead to suboptimal posterior contraction, and would like to see whether sacrificing convergence rate can improve the coverage performance of credible intervals. The results are displayed by green curves, and shows that there is no improvement comparing to the black curves (which corresponds to $t$-prior with $\tau=(s/n)^{c}$).

\vskip 0.2in
\bibliographystyle{plainnat}
\bibliography{ref}